\documentclass[a4paper,reqno, 12pt]{amsart}
\usepackage[left=2.8cm,right=2.8cm,top=2.86cm,bottom=2.86cm]{geometry}
\usepackage{amsmath,amssymb,latexsym,esint,cite,mathrsfs}
\usepackage[ugly]{nicefrac}
\usepackage{amsmath}
\usepackage{amssymb}
\usepackage{latexsym}
\usepackage{amsthm}
\usepackage{hyperref}
\usepackage{array}
\usepackage{color}
\hypersetup{linkcolor=blue, colorlinks=true ,citecolor = red}
\DeclareMathOperator*{\essinf}{ess\,inf}
\DeclareMathOperator*{\esssup}{ess\,sup}
\DeclareMathOperator*{\loc}{loc}
\newcolumntype{M}[1]{>{\centering\arraybackslash}m{#1}}
\newcolumntype{N}{@{}m{0pt}@{}}
\newcommand*\diff{\mathop{}\!\mathrm{d}}
\begin{document}
	\title[  The weighted \MakeLowercase {p}-Laplacian in symmetric domains]
	{On the eigenvalue problem involving the weighted \MakeLowercase{$p$}-Laplacian in radially symmetric domains}
	
	\author[P. Dr\'abek, K. Ho \& A. Sarkar]
	{Pavel Dr\'{a}bek, Ky Ho and Abhishek Sarkar}
	\address{Pavel Dr\'{a}bek \newline
		Department of Mathematics, University of West Bohemia, Univerzitn\'i 8, 306 14 Plze\v{n}, Czech Republic}
	\email{pdrabek@kma.zcu.cz}
	\address{Ky Ho \newline
		Department of Mathematics, Faculty of Sciences, Nong Lam University, Linh Trung Ward, Thu Duc District, Ho Chi Minh City, Vietnam}
	\email{hnky81@gmail.com}
	\address{Abhishek Sarkar \newline
		NTIS, University of West Bohemia, Technick\'a 8, 306 14 Plze\v{n}, Czech Republic}
	\email{sarkara@ntis.zcu.cz}
	
	
	\begin{abstract}
		We investigate the following  eigenvalue problem
		\begin{align*}
		\begin{cases}
		-\operatorname{div}\left( L(x) |\nabla u| ^{p-2}\nabla u\right)=\lambda K(x)|u|^{p-2}u \quad \text{in } A_{R_1}^{R_2} ,\\
		u=0\quad \text{on } \partial A_{R_1}^{R_2} ,
		\end{cases}
		\end{align*} 
		where $A_{R_1}^{R_2}:=\{x\in\mathbb{R}^N: R_1<|x|<R_2\}$ $(0< R_1<R_2\leq\infty)$, $\lambda>0$ is a parameter, the weights $L$ and $K$ are measurable with $L$ positive a.e. in $A_{R_1}^{R_2}$ and $K$ possibly sign-changing in $A_{R_1}^{R_2}$. We prove the existence of the first eigenpair and discuss the regularity and positiveness of eigenfunctions. 
		 The asymptotic estimates for $u(x)$ and $\nabla u(x)$ as $|x|\to R_1^+$ or $R_2^-$ are also investigated. 
	\end{abstract}
	
	\maketitle
	\numberwithin{equation}{section}
	\newtheorem{theorem}{Theorem}[section]
	\newtheorem{lemma}[theorem]{Lemma}
	\newtheorem{proposition}[theorem]{Proposition}
	\newtheorem{corollary}[theorem]{Corollary}
	\newtheorem{definition}[theorem]{Definition}
	\newtheorem{example}[theorem]{Example}
	\newtheorem{remark}[theorem]{Remark}
	\allowdisplaybreaks

\section{Introduction and main results}\label{introduction}
In this paper we investigate the following eigenvalue problem
\begin{align}\label{1.1}
\begin{cases}
-\operatorname{div}\left( L(x) |\nabla u| ^{p-2}\nabla u\right)=\lambda K(x)|u|^{p-2}u \quad \text{in } A_{R_1}^{R_2} ,\\
u=0\quad \text{on } \partial A_{R_1}^{R_2} ,
\end{cases}
\end{align}
where the weight $L$ is measurable and positive a.e. in  $A_{R_1}^{R_2}:=\{x\in\mathbb{R}^N: R_1<|x|<R_2\}$ $(0< R_1<R_2\leq\infty)$  such that $L\in L^1_{\loc}(A_{R_1}^{R_2} );$  the weight $K$ is measurable in  $A_{R_1}^{R_2}$ such that $\operatorname{meas}\{x\in A_{R_1}^{R_2}: K(x)>0\}>0;$ $\lambda$ is a spectral parameter. For the notational convenience we denote the operator $\operatorname{div}\left( L(x) |\nabla u| ^{p-2}\nabla u\right)$ by $\Delta_{p,L}$ and by $|S|$ we denote the Lebesgue measure of $S \subset \mathbb{R}^N$. We note that $K$ might change the sign in $A_{R_1}^{R_2}.$
\begin{itemize}
	\item[($\mathrm{A}$)]  there exist functions $v,w$ measurable and positive a.e. in $(R_1,R_2)$, such that $v^{-\frac{1}{p-1}},w\in L^1_{\loc}(R_1,R_2)$ and
	\begin{itemize}
		\item[(i)] $P(r):=\min\left\{\left(\int_{R_1}^{r}\rho^{1-p'}(\tau)\diff \tau\right)^{p-1},\left(\int_{r}^{R_2}\rho^{1-p'}(\tau)\diff \tau\right)^{p-1}\right\}<\infty$ for all $r\in (R_1,R_2)$ and $\int_{R_1}^{R_2}P(r)\sigma(r)\diff r<\infty,$ where $p':=\frac{p}{p-1},$ $\rho(r):=r^{N-1}v(r)$ and $\sigma(r):=r^{N-1}w(r);$
		\item[(ii)] $L(x)\geq v(|x|)$ and $|K(x)|\leq w(|x|)$ for a.e. $x\in A_{R_1}^{R_2}.$
	\end{itemize}
	\end{itemize} 
\par Equation \eqref{1.1}, which contains  \emph{weighted $p$-Laplacian} operator $\Delta_{p,L}$, describes several important phenomena which arise in Mathematical Physics, Riemannian geometry, Astrophysics, study of non-Newtonian fluids, subsonic motion of gases etc. (see e.g., \cite{Mitidieri, Wang}). 
A weighted second order linear differential operator was basically introduced by  Murthy and Stampacchia \cite{Murthy}, being then extended to higher order linear weighted elliptic operators in the 80s and quasilinear elliptic equations including the weighted $p$-Laplacian in the 90s (see Dr\'{a}bek et al. \cite{Dra-Kuf-Nic}).

\par  The problem \eqref{1.1} in case of bounded domains or $\mathbb{R}^N$, was comprehensively investigated in \cite{Dra-Kuf-Nic}, with suitable weights, and later studied by many authors, we mention Le-Schmitt \cite{Les2}, L\^{e}-Schmitt \cite{Les1},  and references therein.

\par The weighted $p$-Laplacian eigenvalue problem in case of unbounded domains has got attention in the last two decades. In \cite{Mon-Rad, Per-Pucci}, authors studied existence of an eigensolution with nonnegative weights on the right hand side for a nonlinear eigenvalue problem with mixed boundary condition. For an exterior domain $B_1^c$, the complement of the closed unit ball in $\mathbb{R}^N$ ($N\geq 2$), Anoop et al. \cite{Anoop.CV, Anoop.NA} studied the eigenvalue problem \eqref{1.1} with $L(x)\equiv 1$ and the weight $K$ satisfying the following condition
\begin{itemize}
	\item[(ADS)]  $K\in L_{\loc}^1(B_1^c)$, $\operatorname{meas}\{x\in B_1^c: K(x)>0\}>0$ and there exists a positive function $w$ such that 
	\begin{itemize}
		\item[(i)]  $w\in \begin{cases}
		L^1((1,\infty);r^{p-1}),\ p\ne N,\\
		L^1((1,\infty);[r\log r]^{N-1}),\ p=N;
		\end{cases}$
		\item[(ii)] $|K(x)|\leq w(|x|)$ for a.e. $x\in B_1^c.$ 
		
	\end{itemize}  
	
\end{itemize}
The authors proved the existence of a principal eigenvalue and discussed positivity and regularity of associated eigenfunctions when $K$ satisfies some additional assumptions. It is worth mentioning that they allowed also the case $p\geq N$ and $K$ possibly changing sign.

 Another interesting aspect of qualitative properties is the behavior of solutions towards the boundary. The asymptotic estimates for solutions to problem \eqref{1.1} in exterior domains with $L(x)\equiv 1$ was obtained by several authors (see e.g., \cite{Chhetri-Drabek, Anoop.CV}). However, very few works deal with such kind of estimates for the weighted $p$-Laplacian. In the open ball $B_R$ of radius $R$ $(0<R\leq\infty)$ centered at the origin with the convention that $B_R:=\mathbb{R}^N$ when $R=\infty,$ the authors in \cite{Dra-Kuf-Kul, Oscar-Drabek} recently obtained the asymptotic estimates for 
solutions to \eqref{1.1} with radially symmetric weights $L(x)=v(|x|)$ and $K(x)=w(|x|)$ satisfying the following condition introduced in the book by Opic and Kufner \cite{Opic-Kufner}: 
\begin{itemize}
	\item [(OK)] $\begin{cases}
	\text{either }\ \left(\int_{a}^{r}\sigma(\tau)\diff \tau\right)\left(\int_{r}^{b}\rho^{1-p'}(\tau)\diff \tau\right)^{p-1}\to 0\ \text{as}\ r\to a^+,b^-,\\
	\text{or }\ \left(\int_{r}^{b}\sigma(\tau)\diff \tau\right)\left(\int_{a}^{r}\rho^{1-p'}(\tau)\diff \tau\right)^{p-1}\to 0\ \text{as}\ r\to a^+,b^-, -\infty \leq a < b \leq \infty,
	\end{cases}$
\end{itemize}
with $a=0$ and $b=R$.
\par
The goal of this paper is twofold. First, we investigate the eigenvalue problem \eqref{1.1} with the weights $L,K$ possibly not bounded and/or not separated away from zero in a general radially symmetric domain $A_{R_1}^{R_2}.$ Second, we obtain the asymptotic estimates for solutions to problem \eqref{1.1} when the weights are radially symmetric. As in \cite{Anoop.CV}, there is no restriction on the dimension $N$ in terms of $p.$ We emphasize that for simplicity and clarity of statements of our results we are only concerned with two types of domains:  annulus ($0<R_1<R_2<\infty$) and exterior of the ball of radius $R_1$ ($0<R_1<R_2=\infty$). In fact, 
	some of our results also covers other two types of radially symmetric domains: bounded balls $B_R$ ($0<R<\infty$) and the entire space $\mathbb{R}^N$ (see Remarks~\ref{domains} and \ref{problems_on_balls}). 

The novelty of this paper consists in considering \eqref{1.1} with new condition on the weights. 
 Even when $L(x)=v(|x|)\equiv 1$, the condition $\textup{(A)}$ for the weight $K$ is slightly weaker than the condition $\textup{(ADS)}$ introduced in \cite{Anoop.CV} (see Remark~\ref{2.Rmk.Compare_with_ADS} in Section~\ref{Sec.Preliminaries}). It is worth mentioning that there are weights $v,w$ which satisfy $\textup{(A)}$ but do not satisfy $\textup{(OK)}$ (see Remark~\ref{2.Rmk.compare_with_(OK)} in Section~\ref{Sec.Preliminaries}). We confess that we are not aware of weights $v$ and $w$ satisfying (OK) but not (A). Hence the class of weights satisfying $\textup{(A)}$ is a complement of the class of weights satisfying $\textup{(OK)}$ in order to study \eqref{1.1} with radially symmetric weights. 

We look for solutions of \eqref{1.1} in the space $\mathcal{D}_{0}^{1,p}(A_{R_1}^{R_2};L),$ which is the completion of $C_c^1(A_{R_1}^{R_2})$ ($C^1$ functions with compact support) with respect to the norm
$$\|u\|:=\left(\int_{A_{R_1}^{R_2}}L(x)|\nabla u|^p\diff x\right)^{1/p}.$$ 
We note that $\mathcal{D}_{0}^{1,p}(A_{R_1}^{R_2};L)$ is well defined uniformly convex Banach space under the assumption $\textup{(A)}$ (see Theorem~\ref{2.theorem.embedding} in Section~\ref{Sec.Preliminaries}). Moreover, we will prove in Section~\ref{Sec.Preliminaries} that if $\textup{(A)}$ holds and $L^{-s}\in L^1_{\loc} (A_{R_1}^{R_2})$ for some $s\in (\frac{N}{p},\infty)\cap [\frac{1}{p-1},\infty),$ then $\mathcal{D}_{0}^{1,p}(A_{R_1}^{R_2};L)$ is compactly embedded in $L^p(A_{R_1}^{R_2};w),$ the space of measurable functions $u$ such that $\int_{A_{R_1}^{R_2}}w(|x|)|u|^p \diff x<\infty$ (see Theorem~\ref{2.theorem.compact_embedding}). 
\begin{definition}\rm
By a \emph{(weak) solution} of problem \eqref{1.1}, we mean a function $u\in \mathcal{D}_{0}^{1,p}(A_{R_1}^{R_2};L)$ such that
	\begin{equation*}     
	\int_{A_{R_1}^{R_2}}L(x)|\nabla u|^{p-2}\nabla u\cdot \nabla v \diff x=\lambda \int_{A_{R_1}^{R_2}}K(x)|u|^{p-2}uv\diff x, \quad \forall v\in  \mathcal{D}_{0}^{1,p}(A_{R_1}^{R_2};L).   
\end{equation*}
\end{definition}
 If problem \eqref{1.1} has a nontrivial solution $u$ then $\lambda$ is called an \emph{eigenvalue} of $-\Delta_{p,L}$ in $A_{R_1}^{R_2}$ related to the weight $K$ (an eigenvalue, for short) and such a solution $u$ is called an \emph{eigenfunction} corresponding to the eigenvalue $\lambda.$\\
Define
\begin{equation}\label{def.lamda1}
\lambda_1:=\inf \left\{\int_{A_{R_1}^{R_2}}L(x)|\nabla u|^p\diff x:\ u\in \mathcal{D}_{0}^{1,p}(A_{R_1}^{R_2};L),\int_{A_{R_1}^{R_2}}K(x)|u|^p\diff x=1\right\}.
\end{equation}
We state our first main result of the existence of a principal eigenvalue and its simplicity.
\begin{theorem}[Principal eigenpair]\label{Theore.eigenpair}
	Assume that $\mathrm{(A)}$ holds and $L^{-s} \in L^1_{\loc}(A_{R_1}^{R_2})$ for some $s\in (\frac{N}{p},\infty)\cap [\frac{1}{p-1},\infty).$ Then $\lambda_1>0$ and $\lambda_1$ is a simple eigenvalue of \eqref{1.1}. Moreover  $\lambda_1$ is  achieved at an eigenfunction $\varphi_1,$ which is positive a.e. in $A_{R_1}^{R_2}.$ 
\end{theorem}
Next, we state our results on the boundedness of solutions to problem \eqref{1.1} that will be utilized to obtain the $C^1$ regularity of solutions. The following theorems  show that all eigenfunctions to eigenvalue problem \eqref{1.1} are locally bounded in $A_{R_1}^{R_2}$ if the weights satisfy some additional assumptions. In fact, in Section~\ref{Sec.Boundedness} we obtain the boundedness of solutions for a more general nonlinear term (see Theorem~\ref{apriorimainthm}) via the De Giorgi type iteration technique. In the sequel, for $\alpha>0$ we use the convention that $\frac{\alpha}{0}:=\infty$ and define $p_\alpha:=\frac{p\alpha}{\alpha+1}$ and  $\alpha^\ast:=\begin{cases}
\frac{N\alpha}{N-\alpha}& \text{if\ \ } \alpha<N,\\
\infty& \text{if\ \ } \alpha\geq N.
\end{cases}$ 
\begin{theorem}[Boundedness I]\label{Theorem.local.behaviorI}
	Assume that $\mathrm{(A)}$ holds.  Assume in addition that $L^{-s},L^{\frac{q}{q-p}}, |K|^{\frac{q}{q-p}}\in L^1(A_{R_1}^{R_1+2\epsilon})$ for some $\epsilon\in (0,\frac{R_2-R_1}{2})$, $s\in (\frac{N}{p},\infty)\cap [\frac{1}{p-1},\infty)$ and $q\in [p,p_s^\ast).$   Then  for any solution $u$ of problem \eqref{1.1} we have $u\in L^q(A_{R_1}^{R_1+2\epsilon})\cap L^\infty\big(A_{R_1}^{R_1+\epsilon}\big)$ and there exist $C>0$ and $\mu>0$ (independent of $u$) such that 
	\begin{equation*} 
	\|u\|_{L^\infty\big(A_{R_1}^{R_1+\epsilon}\big)}\leq C\left[1+\left(\int_{A_{R_1}^{R_1+2\epsilon}}|u|^q\diff x\right)^{\mu}\right].
	\end{equation*}
\end{theorem}
\begin{theorem}[Boundedness II]\label{Theorem.local.behaviorII}
		Assume that $\mathrm{(A)}$ holds. Assume in addition that $L^{-s},L^{\frac{q}{q-p}}, |K|^{\frac{q}{q-p}}\in L^1(B(x_0,r_0))$ for some ball $B(x_0,r_0)\subset A_{R_1}^{R_2}$, 
		 $s\in (\frac{N}{p},\infty)\cap [\frac{1}{p-1},\infty)$ and $q\in [p,p_s^\ast).$  Then  for any given $\mu\in (0,1-\frac{q}{p_s^\ast})$, there exists $C=C(\mu,r_0)>0$ such that for any solution $u$ of problem \eqref{1.1} we have $u\in L^q(B(x_0,r_0))\cap L^\infty \big(B\big(x_0,\frac{r_0}{2}\big)\big)$ and 
	\begin{equation} \label{localmainestI}
	\|u\|_{L^\infty\big(B\big(x_0,\frac{r_0}{2}\big)\big)}\leq CM_{L,K}\left(\int_{B(x_0,r_0)}|u|^q\diff x\right)^{\frac{1}{q}}.
	\end{equation}
	Here $$M_{L,K}:=\left(\int_{B(x_0,r_0)}L^{-s}(x)\diff x\right)^{\frac{1}{\mu sp}}\left[\|L\|_{L^{\frac{q}{q-p}}(B(x_0,r_0))}+\|K\|_{L^{\frac{q}{q-p}}(B(x_0,r_0))}\right]^{\frac{1}{\mu p}}.$$ In particular, if $L^{-s},L^{\frac{q}{q-p}}$ and $ |K|^{\frac{q}{q-p}}\in L_{\loc}^1(A_{R_1}^{R_2}),$ then $u \in L_{\loc}^\infty(A_{R_1}^{R_2}).$
\end{theorem}

%
%

We now discuss certain smoothness properties of eigenfunctions. In the sequel, for an open set $\Omega$ in $\mathbb{R}^N$ we denote by $W^1(\Omega)$ the set of all $u\in L_{\loc}^1(\Omega)$ such that weak derivatives $\frac{\partial u}{\partial x_i}\  (i=1,\cdots, N)$ exist in $\Omega$. We first have the $C^1$ regularity of eigenfunctions in $A_{R_1}^{R_2}.$ 

\begin{theorem} \label{Theorem.C1-regularity1}	Assume that $\mathrm{(A)}$ holds. Assume in addition that $L\in W^1(A_{R_1}^{R_2}),$ $\underset{x\in A_{r_1}^{r_2}}{\essinf}\ L(x)>0$ for any $R_1<r_1<r_2<R_2,$ $L,K\in L_{\loc}^{\frac{q}{q-p}}(A_{R_1}^{R_2})$ for some $q\in [p,p_s^\ast),$ and $|\frac{K}{L}|+|\frac{\nabla L}{L}|^p\in L^{\widetilde{q}}_{\loc}(A_{R_1}^{R_2})$ for some $\widetilde{q}>\frac{Np}{p-1}.$ Then for a (weak) solution  $u$ of \eqref{1.1}, we have $u\in C^1(A_{R_1}^{R_2}).$ 
\end{theorem}
The next result provides the regularity of eigenfunctions up to the inner boundary.
\begin{theorem} \label{Theorem.C1-regularity2} In addition to the assumptions of  Theorem~\ref{Theorem.C1-regularity1}, we also assume that $\underset{x\in A_{R_1}^{R_1+\epsilon}}{\essinf}\ L(x)>0$, $L,K\in L^{\frac{q}{q-p}}(A_{R_1}^{R_1+\epsilon})$ and $|\frac{K}{L}|+|\frac{\nabla L}{L}|\in L^\infty(A_{R_1}^{R_1+\epsilon})$ for some $\epsilon\in (0,R_2-R_1).$ Then for a (weak) solution  $u$ of \eqref{1.1} and $R\in (R_1,R_2),$  $u\in C^{1,\alpha(R)}(\overline{A_{R_1}^R})$ for some $\alpha(R)\in (0,1).$
\end{theorem}

In view of the $C^1$ regularity of eigenfunctions above and the strong maximum principle we have the following result.
\begin{theorem}\label{Theorem.everywhere-positivity}
	Assume that $\mathrm{(A)}$ holds. Assume in addition that  $K\in L^\infty_{\loc}(A_{R_1}^{R_2})$ and $L\in C^1_{\loc}(A_{R_1}^{R_2})$ such that  $\underset{x\in A_{r_1}^{r_2}}{\essinf}\ L(x)>0$ for all $R_1<r_1<r_2<R_2.$  Let $u$ be a nonnegative eigenfunction of \eqref{1.1}. 
Then, $u\in C^1(A_{R_1}^{R_2})$ and $u>0$ everywhere in $A_{R_1}^{R_2}.$
\end{theorem}

Finally, we discuss the decay of the solutions to problem \eqref{1.1} when $|x|\to R_1^+$ or $R_2^-$, that is important to obtain the asymptotic estimates near the boundary. Using the local behavior obtained in Theorem~\ref{Theorem.local.behaviorII} we can obtain the decay of the solutions when $R_2=\infty$ and $L$ is non-degenerate at infinity.

\begin{corollary}\label{Coro.decay.at.infinity.2}
	Assume that $1<p<N,$ $R_2=\infty$ and $\mathrm{(A)}$ holds.   
	Assume in addition that there exists $R\in (R_1,\infty)$ such that $\underset{x \in B_R^c}{\essinf}\ L(x)>0,$  $L,K\in L_{\loc}^{\frac{q}{q-p}}(B_{R}^c)$ for some $q\in [p,p^\ast)$ and 
	\begin{equation*}
	\esssup_{x \in B_R^c} \int_{B(x,r_0)} \bigg[L^{\frac{q}{q-p}}(y)+|K(y)|^{\frac{q}{q-p}}\bigg] \diff y < \infty,
	\end{equation*}
	for some $r_0 \in (0,R-R_1)$. Then, for any solution $u$ to problem \eqref{1.1}, we have $u(x) \to 0$ uniformly as $|x| \to \infty$.
\end{corollary}

The decay of solutions when $|x|\to R_1^+$ follows immediately if $u\in C^{1,\alpha}(\overline{A_{R_1}^R})$ for some $R>R_1$ and $\alpha\in (0,1).$

\begin{corollary}\label{Coro.decay.at.R1} Under the assumption of Theorem~\ref{Theorem.C1-regularity2}, for any solution $u$ of \eqref{1.1}, we have $u(x)\to 0$ as $|x|\to R_1^+.$
\end{corollary}
Next, we draw our attention to prove asymptotic behavior of a $C^1$ radially symmetric solution $u(x)=u(|x|)$ and its gradient to equation 
\begin{equation}\label{1.2}
-\operatorname{div}\left( v(|x|) |\nabla u| ^{p-2}\nabla u\right)=\lambda w(|x|)|u|^{p-2}u \quad \text{in } A_{R_1}^{R_2},
\end{equation}
as $|x|\to R_1^+$ 
or $|x|\to R_2^-$ if $u(x)\to 0$ as $|x|\to R_1^+$ and $|x|\to R_2^-$. We assume
  \begin{itemize}
  	\item[($\mathrm{W}$)]  $v,w$  are positive a.e. in $(R_1,R_2)$ such that $v$ (resp. $w$) is continuous (resp. measurable) in $(R_1,R_2)$ satisfying $v^{-\frac{1}{p-1}}\in L^1_{\loc}(R_1,R_2)$ (resp. $w\in L^1_{\loc}(R_1,R_2)$).
  \end{itemize} 
Note that a similar problem in the case of a ball $B_R$ ($0<R\leq\infty$) was investigated in \cite{Dra-Kuf-Nic}. We write $u(R_1)=\lim_{r\to R_1^+}u(r)$ and $u(R_2)=\lim_{r\to R_2^-}u(r).$ Clearly, if $u(x)=u(|x|)\in C^1(A_{R_1}^{R_2})$ is a radially symmetric solution to problem \eqref{1.2} with $u(x)\to 0$ as $|x|\to R_1^+$ and $|x|\to R_2^-$ , then $u\in C^1(R_1,R_2)$ satisfies
\begin{equation}\label{ode}
-\left(\rho(r)|u'(r)|^{p-2}u'(r)\right)'=\lambda \sigma(r)|u(r)|^{p-2}u(r)\quad \text{in}\ (R_1,R_2)
\end{equation}  
and $u(R_1)=u(R_2)=0$. In two Theorems \ref{radialsoldecayleft} and \ref{radialsoldecayright}, we show that if the conditions on weights are made stronger than (A) near $R_1$ and  $R_2$ (see Remark~\ref{Rmk.non-oscillation}) then solutions obey certain decay properties.
Namely, we assume 
\begin{itemize}
	\item[$\mathrm{(A_{\epsilon,L})}$] \ \  there exists $\xi \in (R_1,R_2)$ such that $\rho^{1-p'}\in L^1(R_1;\xi)$, and there exist $\epsilon\in (0,p-1)$ and $C>0$ such that
	\begin{equation*}
	\left(\int_{r}^{\xi}\sigma(\tau)\diff \tau \right)\left(\int_{R_1}^{r}\rho^{1-p'}(\tau)\diff \tau\right)^{\epsilon}<C,\quad \forall r\in(R_1,\xi);
	\end{equation*}
\end{itemize} 
\begin{itemize}
	\item[$\mathrm{(A_{\epsilon,R})}$]\ \  there exists $\xi \in (R_1,R_2)$ such that $\rho^{1-p'}\in L^1(\xi,R_2)$, and there exist $\epsilon\in (0,p-1)$ and $C>0$ such that
	\begin{equation*}
	\left(\int_{\xi}^{r}\sigma(\tau)\diff \tau \right)\left(\int_{r}^{R_2}\rho^{1-p'}(\tau)\diff \tau\right)^{\epsilon}<C,\quad \forall r\in(\xi,R_2).
	\end{equation*}
\end{itemize} 

\begin{theorem} \label{radialsoldecayleft}
	Assume that $\mathrm{(W)}$ and $\mathrm{\big(A_{\epsilon,L}\big)}$ hold. Then for a radially symmetric solution $u(x)=u(|x|)\in C^1(A_{R_1}^{R_2})$ to problem \eqref{1.2} satisfying $u(R_1)=u(R_2)=0,$ there exist $a\in(R_1,R_2)$ and $0<C_1<C_2,0<\widetilde{C}_1<\widetilde{C}_2$ such that
\begin{equation} \label{estextsol2}
	C_1\int_{R_1}^{r}\rho^{1-p'}(\tau)\diff \tau\leq |u(r)|\leq C_2\int_{R_1}^{r}\rho^{1-p'}(\tau)\diff \tau,\quad \forall r\in (R_1,a), \end{equation}
and 
\begin{equation}\label{estextsolder2}
	\widetilde{C}_1\rho^{1-p'}(r)\leq |u'(r)|\leq \widetilde{C}_2\rho^{1-p'}(r),\quad \forall r\in (R_1,a). 
\end{equation}
\end{theorem} 

\begin{theorem}\label{radialsoldecayright}
	Assume that $\mathrm{(W)}$ and $\mathrm{(A_{\epsilon,R})}$ hold. Then for a radially symmetric solution $u(x)=u(|x|)\in C^1(A_{R_1}^{R_2})$ to problem \eqref{1.2} satisfying $u(R_1)=u(R_2)=0,$ there exist $b\in(R_1,R_2)$ and $0<C_1<C_2,0<\widetilde{C}_1<\widetilde{C}_2$ such that
	$$C_1\int_{r}^{R_2}\rho^{1-p'}(\tau)\diff \tau\leq |u(r)|\leq C_2\int_{r}^{R_2}\rho^{1-p'}(\tau)\diff \tau,\quad \forall r\in (b,R_2),$$
	and 
	$$\widetilde{C}_1\rho^{1-p'}(r)\leq |u'(r)|\leq \widetilde{C}_2\rho^{1-p'}(r),\quad \forall r\in (b,R_2).$$
\end{theorem} 

The rest of the paper is organized as follows. In Section~\ref{Sec.Preliminaries},
we obtain some useful embeddings of the weighted Sobolev spaces into weighted Lebesgue spaces defined earlier. In Section~\ref{Sec.EigenvalueProblem}, we prove the existence of the least positive eigenvalue and the corresponding 
positive eigenfunction associated to problem \eqref{1.1}. The simplicity of such an eigenvalue is also discussed in this section. 
  Section~\ref{Sec.Boundedness} deals with boundedness, smoothness and decay of solutions to problem \eqref{1.1}. Section~\ref{Sec.Behavior} is devoted to the investigation of the behavior of $u(x)$ and $\nabla u(x)$ as $|x|\to R_1^+$ or $R_2^-,$ in the case of radially symmetric solutions. Finally, we provide a few concrete examples of weights $L$ and $K$ to illustrate our results in Section~\ref{Sec.Applications}.


\section {Weighted spaces}\label{Sec.Preliminaries}
In this section we will obtain embeddings of certain weighted spaces and other properties. 
In what follows denote by $S_1$ the unit sphere $\{x\in\mathbb{R}^N: |x|=1\}$ and for a function $u$ defined on $A_{R_1}^{R_2}$, 
we write $u(x)=u(r,\omega),$ where $r=|x|$ and $\omega=x/r.$ First, we prove the following continuous embedding.
\begin{theorem}\label{2.theorem.embedding}
	Assume that {\rm(A)} holds. Then, we have the following embedding
	$$\mathcal{D}_{0}^{1,p}(A_{R_1}^{R_2};L) \hookrightarrow L^p(A_{R_1}^{R_2};w).$$
\end{theorem}
\begin{proof}
Let $u\in C_c^1(A_{R_1}^{R_2})$ and $r\in (R_1,R_2).$ If $\int_{R_1}^{r}\rho^{1-p'}(\tau)\diff \tau<\infty,$ using H\"older's inequality we estimate
\begin{align*}
|u(r,\omega)|&=\left|\int_{R_1}^{r}\frac{\partial u}{\partial \tau}(\tau,\omega)d\tau\right|=\left|\int_{R_1}^{r}\rho^{-\frac{1}{p}}(\tau)\tau^{\frac{N-1}{p}}v^{\frac{1}{p}}(\tau)\frac{\partial u}{\partial \tau}(\tau,\omega)\diff \tau\right|\\
&\leq \left(\int_{R_1}^{r}\rho^{1-p'}(\tau)\diff \tau\right)^{\frac{1}{p'}}\left(\int_{R_1}^{R_2}\tau^{N-1}v(\tau)\left|\frac{\partial u}{\partial \tau}(\tau,\omega)\right|^p\diff \tau\right)^{\frac{1}{p}}.
\end{align*}
Hence, 
	$$|u(r,\omega)|^p\leq \left(\int_{R_1}^{r}\rho^{1-p'}(\tau)\diff \tau\right)^{p-1}\left(\int_{R_1}^{R_2}\tau^{N-1}v(\tau)\left|\frac{\partial u}{\partial \tau}(\tau,\omega)\right|^p\diff \tau\right).$$
Analogously, if $\int_{r}^{R_2}\rho^{1-p'}(\tau)\diff\tau<\infty,$ we have
$$|u(r,\omega)|^p\leq \left(\int_{r}^{R_2}\rho^{1-p'}(\tau)\diff \tau\right)^{p-1}\left(\int_{R_1}^{R_2}\tau^{N-1}v(\tau)\left|\frac{\partial u}{\partial \tau}(\tau,\omega)\right|^p\diff \tau\right).$$ 	
In either case, we obtain
$$|u(r,\omega)|^p\leq P(r)\int_{R_1}^{R_2}\tau^{N-1}v(\tau)\left|\frac{\partial u}{\partial \tau}(\tau,\omega)\right|^p\diff \tau.$$ 	
Hence,
\begin{align*}
\int_{S_1}|u(r,\omega)|^p\diff \omega\leq P(r)&\int_{S_1}\int_{R_1}^{R_2}\tau^{N-1}v(\tau)\left|\frac{\partial u}{\partial \tau}(\tau,\omega)\right|^p\diff \tau \diff \omega\\
&=P(r)\int_{A_{R_1}^{R_2}}v(|x|)|\nabla u(x)|^pdx.
\end{align*}
Combining this with the assumption $\mathrm{(A)}$ (ii),  we get
\begin{equation}\label{2.theorem1.est.u}
\int_{S_1}|u(r,\omega)|^p\diff \omega\leq \|u\|^pP(r),\quad \forall r\in (R_1,R_2)\  \text{and}\ \forall u\in C_c^1(A_{R_1}^{R_2}).
\end{equation}
From this we deduce

$$\int_{R_1}^{R_2}r^{N-1}w(r)\int_{S_1}|u(r,\omega)|^p\diff \omega \diff r\leq \|u\|^p\int_{R_1}^{R_2}r^{N-1}w(r)P(r)\diff r.$$
That is,
\begin{equation}\label{2.theorem1.est.norm}
\|u\|_{L^p\left(A_{R_1}^{R_2};w\right)}\leq C\|u\|,\quad \forall u\in C_c^1(A_{R_1}^{R_2}),
\end{equation}
where $C:=\left(\int_{R_1}^{R_2}P(r)\sigma(r)\diff r\right)^{\frac{1}{p}}.$ By the density of $C_c^1(A_{R_1}^{R_2})$ in $\mathcal{D}_{0}^{1,p}(A_{R_1}^{R_2};L)$ we obtain \eqref{2.theorem1.est.norm} for all $u\in \mathcal{D}_{0}^{1,p}(A_{R_1}^{R_2};L)$ and it infers the continuity of the embedding. 
\end{proof}

In what follows, for a normed space $(X,\|\cdot\|_X)$ of functions $u: \Omega\to \mathbb{R}$ with $\Omega\subseteq A_{R_1}^{R_2}$ such that $u|_\Omega\in X$ for all $u\in \mathcal{D}_{0}^{1,p}(A_{R_1}^{R_2};L)$, we still denote $\mathcal{D}_{0}^{1,p}(A_{R_1}^{R_2};L) \hookrightarrow X$ if there is a constant $C>0$ such that
$$\|u|_\Omega\|_X\leq C\|u\|,\quad \forall u\in \mathcal{D}_{0}^{1,p}(A_{R_1}^{R_2};L).$$
In fact such an embedding is not an injective map. In this sense the following embeddings are deduced from Theorem~\ref{2.theorem.embedding}

\begin{corollary}\label{2.corollary.local_embeddings}
	Assume that the weight $L$ satisfies
	\begin{itemize}
		\item[(A1)] $L(x) \geq v(|x|)>0$ for a.e. $x \in A_{R_1}^{R_2}$, where $v$ is measurable in $(R_1,R_2)$ such that $v, v^{-\frac{1}{p-1}}\in L^1_{\loc}(R_1,R_2)$ and $P(r)<\infty$ for all $r\in (R_1,R_2),$ where $P$ is defined as in {\rm(A)}.
	\end{itemize}
For any given $R_1<r_1<r_2<R_2$, the following embeddings hold:
	\begin{itemize}
		\item[(i)] $\mathcal{D}_{0}^{1,p}(A_{R_1}^{R_2};L)\hookrightarrow L^p(A_{r_1}^{r_2});$
		\item[(ii)] $\mathcal{D}_{0}^{1,p}(A_{R_1}^{R_2};L)\hookrightarrow W^{1,p_s}(A_{r_1}^{r_2})$ if $L^{-s} \in L^1(A_{r_1}^{r_2})$ for some $s\in (\frac{N}{p},\infty)\cap [\frac{1}{p-1},\infty);$  
		\item[(iii)] $\mathcal{D}_{0}^{1,p}(A_{R_1}^{R_2};L)\hookrightarrow W^{1,p}(A_{r_1}^{r_2})$ if $\underset{x\in A_{r_1}^{r_2}}{\essinf}\ L(x)>0.$
	\end{itemize}
\end{corollary}
\begin{proof}
	(i) Let $R_1<r_1<r_2<R_2.$  Set $w(r)=P^{-1}(r)(r+1)^{-(N+1)}$ for $r\in (R_1,R_2).$  Then, $w \in L^1_{\mathrm{loc}}(R_1,R_2)$ and we also have 
	$$\int_{R_1}^{R_2}P(r)\sigma(r)\diff r=\int_{R_1}^{R_2}\frac{r^{N-1}}{(r+1)^{N+1}}\diff r<\infty.$$
	From this and the hypothesis $\mathrm{(A_1)}$, we see that (A) holds.
 Thus, applying Theorem~\ref{2.theorem.embedding}, we obtain 
\begin{equation}\label{proof.loc.emb1}
\mathcal{D}_{0}^{1,p}(A_{R_1}^{R_2};L) \hookrightarrow L^p(A_{R_1}^{R_2};w).
\end{equation}	
It is easy to see that, for all $r\in (r_1,r_2),$ we have
$$0< P(r)\leq \min\left\{\left(\int_{R_1}^{r_2}\rho^{1-p'}(\tau)\diff \tau\right)^{p-1},\left(\int_{r_1}^{R_2}\rho^{1-p'}(\tau)\diff \tau\right)^{p-1}\right\}=:C_1<\infty.$$	
Thus,
$$w(r)\geq C_1^{-1}(r_2+1)^{-(N+1)}=:C_2>0,\quad \forall r\in(r_1,r_2),$$	
and hence, 
$$\|u\|_{L^p(A_{r_1}^{r_2})}\leq C_2^{-1/p}\|u\|_{L^p(A_{R_1}^{R_2};w)}, \forall u\in L^p(A_{R_1}^{R_2};w).$$
From this and \eqref{proof.loc.emb1}, it follows $\mathcal{D}_{0}^{1,p}(A_{R_1}^{R_2};L) \hookrightarrow L^p(A_{r_1}^{r_2}).$

\noindent (ii) Let $R_1<r_1<r_2<R_2.$ For $u\in \mathcal{D}_{0}^{1,p}(A_{R_1}^{R_2};L)$ we have
$$\int_{A_{r_1}^{r_2}}|\nabla u|^{p_s}\diff x
\leq \left(\int_{A_{r_1}^{r_2}}L^{-s}(x)\diff x\right)^{\frac{1}{s+1}}\left(\int_{A_{r_1}^{r_2}}L(x)|\nabla u|^p\diff x\right)^{\frac{s}{s+1}}.$$
From this and (i) we deduce the conclusion.

\noindent (iii) The conclusion can be deduced from (i) and the assumption on $L.$
\end{proof}
Next, we show the following compact embedding.
\begin{theorem}\label{2.theorem.compact_embedding}
	Assume that $\mathrm{(A)}$ holds and $L^{-s} \in L^1_{\loc}(A_{R_1}^{R_2})$ for some $s\in (\frac{N}{p},\infty)\cap \big[\frac{1}{p-1},\infty \big).$ We have the following compact embedding
	$$\mathcal{D}_{0}^{1,p}(A_{R_1}^{R_2};L) \hookrightarrow\hookrightarrow L^p(A_{R_1}^{R_2};w).$$
\end{theorem}
\begin{proof}
Let $u_n\rightharpoonup 0$ in $\mathcal{D}_{0}^{1,p}(A_{R_1}^{R_2};L)$ as $n\to \infty.$ We will show that $u_n\to 0$ in $L^p(A_{R_1}^{R_2};w)$ as $n\to \infty.$ To this end we will show that for any $\epsilon>0$, there exists $n_\epsilon\in\mathbb{N}$ such that
\begin{equation}\label{2.limit}
\int_{A_{R_1}^{R_2}}w(|x|)|u_n|^p\diff x<\epsilon^p,\quad \forall n\geq n_\epsilon.
\end{equation}
Without loss of generality we may assume that $\{u_n\}\subset C_c^1(A_{R_1}^{R_2})$ and $\|u_n\|\leq 1$ for all $n\in\mathbb{N}.$  Since $P(r)r^{N-1}w(r)\in L^1(R_1,R_2),$ there exists $g_\epsilon\in C^1_c(R_1,R_2)$ such that
\begin{equation*}
\int_{R_1}^{R_2}|g_\epsilon(r)-P(r)r^{N-1}w(r)|\diff r<\frac{\epsilon^p}{2}.
\end{equation*}
Set $w_\epsilon(r):=P^{-1}(r)r^{1-N}g_\epsilon(r)$ for all $r\in (R_1,R_2).$ Applying \eqref{2.theorem1.est.u} and noticing $\|u_n\|\leq 1$, we estimate
\begin{align}\label{2.theorem2.est.un}
\int_{A_{R_1}^{R_2}}\left|(w-w_\epsilon)(|x|)\right||u_n|^p\diff x&= \int_{R_1}^{R_2}\left|r^{N-1}w(r)-r^{N-1}w_\epsilon(r)\right|\int_{S_1}|u_n(r,\omega)|^p\diff \omega\diff r\notag\\
& \leq \int_{R_1}^{R_2}\left|P(r)r^{N-1}w(r)-g_\epsilon(r)\right|\diff r\notag\\
&<\frac{\epsilon^p}{2},\quad \forall n\in\mathbb{N}.
\end{align}
Let $R_1<r_1<r_2<R_2$ such that $\operatorname{supp}(g_\epsilon)\subset (r_1,r_2).$ Then for a.e. $x\in  A_{r_1}^{r_2},$ we have
$$|w_\epsilon(|x|)|\leq C^{-1}_{r_1r_2} r_1^{1-N}\|g_\epsilon\|_{L^\infty(R_1,R_2)}=:M_\epsilon,$$
where $C_{r_1r_2}:=\min\left\{\left(\int_{R_1}^{r_1}\rho^{1-p'}(\tau)\diff \tau\right)^{p-1},\left(\int_{r_2}^{R_2}\rho^{1-p'}(\tau)\diff \tau\right)^{p-1}\right\}>0.$ Thus, we infer
\begin{equation}\label{2.theorem.est2.un}
\int_{A_{R_1}^{R_2}}\left|w_\epsilon(|x|)\right||u_n|^p\diff x= \int_{A_{r_1}^{r_2}}\left|w_\epsilon(|x|)\right||u_n|^p\diff x\leq M_\epsilon \int_{A_{r_1}^{r_2}}|u_n|^p\diff x,\quad \forall n\in\mathbb{N}.
\end{equation}
By $\mathrm{(A)},$ we have $L^{-\frac{1}{p-1}}\in L^1_{\loc}(A_{R_1}^{R_2})$ and note that this condition guarantees that $\mathcal{D}_{0}^{1,p}(A_{R_1}^{R_2};L)\subset W^1(A_{R_1}^{R_2}).$ By this and the embedding  $\mathcal{D}_{0}^{1,p}(A_{R_1}^{R_2};L)
\hookrightarrow L^p(A_{r_1}^{r_2})$ (see Corollary~\ref{2.corollary.local_embeddings} (i)) we have
\begin{equation}\label{2.theorem.loc.embedding}
\mathcal{D}_{0}^{1,p}(A_{R_1}^{R_2};L)\hookrightarrow W^{1,p}(A_{r_1}^{r_2};L),
\end{equation}
where $W^{1,p}(A_{r_1}^{r_2};L):=\big\{u\in W^1(A_{r_1}^{r_2}): \int_{A_{r_1}^{r_2}}\big[|u|^p+L(x)|\nabla u|^p\big]\diff x<\infty\big\}$ endowed with the norm
$$\|u\|_{W^{1,p}(A_{r_1}^{r_2};L)}:=\left(\int_{A_{r_1}^{r_2}}\big[|u|^p+L(x)|\nabla u|^p\big]\diff x\right)^{\frac{1}{p}}.$$ 
Since $L^{-s}\in L^1(A_{r_1}^{r_2})$ for some $s\in (\frac{N}{p},\infty)\cap [\frac{1}{p-1},\infty),$ we may apply a compact embedding result for weighted Sobolev spaces in \cite[p. 26]{Dra-Kuf-Nic} to obtain
\begin{equation}\label{2.theorem.loc.commpact.embedding}
W^{1,p}(A_{r_1}^{r_2};L)\hookrightarrow\hookrightarrow L^p(A_{r_1}^{r_2}).
\end{equation}
By \eqref{2.theorem.loc.embedding}, we have that $u_n|_{A_{r_1}^{r_2}}\rightharpoonup 0$ in $W^{1,p}(A_{r_1}^{r_2};L)$ as $n\to\infty.$ Combining this with \eqref{2.theorem.loc.commpact.embedding} we get
$u_n|_{A_{r_1}^{r_2}}\to 0$ in $L^p(A_{r_1}^{r_2})$ as $n\to\infty.$ Hence, there exists $n_\epsilon\in\mathbb{N}$ such that
$$M_\epsilon \int_{A_{r_1}^{r_2}}|u_n|^p\diff x<\frac{\epsilon^p}{2},\quad \forall n\geq n_\epsilon.$$
From this and \eqref{2.theorem.est2.un}  we obtain

\begin{equation*}
	\int_{A_{R_1}^{R_2}}\left|w_\epsilon(|x|)\right||u_n|^p\diff x<\frac{\epsilon^p}{2},\quad \forall n\geq n_\epsilon.
\end{equation*}
Finally, combining the last estimate and \eqref{2.theorem2.est.un} we obtain \eqref{2.limit}. Since $\epsilon>0$ was chosen arbitrarily, we get $u_n\to 0$ in $L^p(A_{R_1}^{R_2};w)$ as $n\to \infty$ and the proof is complete. 
\end{proof}
We now present several explicit 
consequences of Theorem~\ref{2.theorem.compact_embedding}. In the next two corollaries, we apply Theorem~\ref{2.theorem.compact_embedding} for $L(x)=v(|x|)$ and write $\mathcal{D}_{0}^{1,p}(A_{R_1}^{R_2};v)$ instead of $\mathcal{D}_{0}^{1,p}(A_{R_1}^{R_2};L).$ As in the assumption $\mathrm{(A)}$, we always denote $\rho(r):=r^{N-1}v(r)$ and $\sigma(r):=r^{N-1}w(r).$
\begin{corollary}  
Let $v,w$ be measurable and positive a.e. in $(R_1,R_2)$ such that $v,v^{-s}\in L^1_{\loc}(R_1,R_2)$ for some $s\in (\frac{N}{p},\infty)\cap [\frac{1}{p-1},\infty)$ and one of the following conditions holds true:
\begin{itemize}
	\item[(I)] there exists $\xi \in (R_1,R_2)$ such that $\int_{\xi}^{R_2}\rho^{1-p'}(r)\diff r <\int_{R_1}^{\xi}\rho^{1-p'}(r)\diff r=\infty $ and   
	$$	\int_{R_1}^{R_2}\left[\int_{r}^{R_2}\rho^{1-p'}(\tau)\diff \tau\right]^{p-1}\sigma(r)\diff r<\infty;$$
	\item[(II)] there exists $\xi \in (R_1,R_2)$ such that $\int_{R_1}^{\xi}\rho^{1-p'}(r)\diff r< \int_{\xi}^{R_2}\rho^{1-p'}(r)\diff r=\infty $ and  
	$$	\int_{R_1}^{R_2}\left[\int_{R_1}^{r}\rho^{1-p'}(\tau)\diff \tau\right]^{p-1}\sigma(r)\diff r<\infty;$$
	\item[(III)] there exists $\xi \in (R_1,R_2)$ such that $\int_{R_1}^{R_2}\rho^{1-p'}(r)\diff r<\infty$ and  
	$$	\int_{R_1}^{\xi}\left[\int_{R_1}^{r}\rho^{1-p'}(\tau)\diff \tau\right]^{p-1}\sigma(r)\diff r+\int_{\xi}^{R_2}\left[\int_{r}^{R_2}\rho^{1-p'}(\tau)\diff \tau\right]^{p-1}\sigma(r)\diff r<\infty.$$
\end{itemize}
Then the following compact embedding holds
	$$\mathcal{D}_{0}^{1,p}(A_{R_1}^{R_2};v) \hookrightarrow\hookrightarrow L^p(A_{R_1}^{R_2};w).$$
\end{corollary}
Finally, we provide a simple special case of Theorem~\ref{2.theorem.compact_embedding}.
\begin{corollary}\label{exteriorembedding}
Let $v,w$ are measurable and positive a.e. in $(R,\infty)$ such that $v,v^{-s}\in L^1_{\loc}(R,\infty)$ for some $R\in (0,\infty)$, $s\in (\frac{N}{p},\infty)\cap [\frac{1}{p-1},\infty)$ and one of the following conditions holds true:
	\begin{itemize}
	\item[($\mathrm{W_1}$)]  there exists $\xi\in(R,\infty)$ such that $\underset{r\geq\xi}\essinf\ v(r)>0,$ $v^{-\frac{1}{p-1}}\in L^1(R,\xi)$ and 
	$$\begin{cases}
	\int_{R}^{\xi}\left[\int_{R}^{r}v^{-\frac{1}{p-1}}(\tau)\diff \tau\right]^{p-1}w(r)\diff r+\int_{\xi}^{\infty}r^{p-1}w(r)\diff r<\infty,\ p\ne N,\\
	\int_{R}^{\xi}\left[\int_{R}^{r}v^{-\frac{1}{N-1}}(\tau)\diff \tau\right]^{N-1}w(r)\diff r+\int_{\xi}^{\infty}[r\log r]^{N-1}w(r)\diff r<\infty,\ p=N;
	\end{cases}$$
	\item[($\mathrm{W_2}$)]  there exists $\xi\in(R,\infty)$ such that $\underset{R\leq r\leq\xi}\essinf\ v(r)>0,$ $\left[r^{N-1}v\right]^{-\frac{1}{p-1}}\in L^1(\xi,\infty),$ and
	$$	\int_{R}^{\xi}(r-R)^{p-1}w(r)\diff r+\int_{\xi}^{\infty}\left[\int_{r}^{\infty}\tau^{-\frac{N-1}{p-1}}v^{-\frac{1}{p-1}}(\tau)\diff \tau\right]^{p-1}r^{N-1}w(r)\diff r<\infty.$$ 
\end{itemize} 
Then, we have the following embedding
		$$\mathcal{D}_{0}^{1,p}(B_R^c;v) \hookrightarrow\hookrightarrow L^p(B_R^c;w).$$
\end{corollary}

\begin{remark}\label{2.Rmk.Compare_with_ADS}\rm
	In particular, $\mathrm{(W_1)}$ is a special case of (A). When $v$ is a constant, say, $v\equiv 1$ and $R=1,$ then $ \mathrm{(W_1)}$ becomes
	\begin{itemize}
		\item[$\mathrm{(W_{1,c})}$] $w\in\begin{cases} 
			L^1((1,\infty);(r-1)^{p-1}),\ p\ne N,\\
			L^1((1,\infty);[r\log r]^{N-1}),\ p=N.
		\end{cases}$
	\end{itemize}
Clearly, a weight $w$ satisfying $\mathrm{(ADS)}$ satisfies also $\mathrm{(W_{1,c})}$. On the other hand, for $-p<\beta\leq -1$ and $p\ne N$ the weight $$w(r)=\begin{cases}
	(r-1)^\beta, \quad 1\leq r\leq 2,\\
	\in L^1((2,\infty); r^{p-1}),
	\end{cases} $$
	satisfies $\mathrm{(W_{1,c})}$ but it does not satisfy $\mathrm{(ADS)}$. Therefore, the condition (A) is weaker than the condition (ADS). 
	
\end{remark}

\begin{remark}\label{2.Rmk.compare_with_(OK)}\rm
It is worth noting that the condition $\mathrm{(OK)}$ does not include $\mathrm{(W_1)}$ and hence, does not include $\mathrm{(A)}$.
 For instance, let $1<p<N,$ $\alpha<p-1,$ $\beta\geq 0,$ $\alpha-p<\alpha_1\leq -1,$ and $-N\leq \beta_1<-p.$ Set
$$v(r)=\begin{cases}
(r-1)^\alpha,\quad 1\leq r\leq 2,\\
\in [1,3^\beta],\quad 2\leq r\leq 3,\\
r^\beta,\quad 3\leq r,
\end{cases}\ \text{and}\ w(r)=\begin{cases}
	(r-1)^{\alpha_1},\quad 1\leq r\leq 2,\\
	\in [3^{\beta_1},1],\quad 2\leq r\leq 3,\\
	r^{\beta_1},\quad 3\leq r.
\end{cases}$$ 
We can verify that $v,w$ satisfy $\mathrm{(W_1)}$ with $R=1$ but $\rho(r)=r^{N-1}v(r)$ and $\sigma(r)=r^{N-1}w(r)$ do not satisfy $\mathrm{(OK)}$  (with $a=1$ and $b =\infty$) since $\int_{1}^{r}\sigma(\tau)\diff \tau=\int_{r}^{\infty}\sigma(\tau)\diff \tau=\infty $ for all $r\in (1,\infty).$ To find $v$ and $w$ which satisfy (OK) but do not satisfy (A) seems to be an open problem. 

\end{remark}
Finally, we state a property of $\mathcal{D}_{0}^{1,p}(A_{R_1}^{R_2};L),$ that will be used in the next sections. In what follows, we denote $u^+=\max\{u,0\}$ and $u^-=-\min\{u,0\}.$
\begin{proposition}\label{prop.(u-k)^+}
If $u\in \mathcal{D}_{0}^{1,p}(A_{R_1}^{R_2};L)$ and $k\ge 0,$ then $(u-k)^+,(u+k)^-\in \mathcal{D}_{0}^{1,p}(A_{R_1}^{R_2};L).$ 
\end{proposition}
\begin{proof}
	Argument is standard and we only sketch the main idea. Since $(u+k)^-=(-u-k)^+$, it suffices to prove that $(u-k)^+\in \mathcal{D}_{0}^{1,p}(A_{R_1}^{R_2};L).$ That is, we prove the existence of a sequence $\{u_n\}\subset C_c^1(A_{R_1}^{R_2})$ such that
	\begin{equation}\label{(u-k)^+}
	\int_{A_{R_1}^{R_2}}L(x)|\nabla u_n-\nabla(u-k)^+|^p\diff x \to 0 \quad \text{as}\quad n\to\infty.
	\end{equation}
	To this end, let $\{\varphi_n\}\subset C_c^1(A_{R_1}^{R_2})$ such that $\|\varphi_n-u\|\to 0$ as $n\to\infty.$ It is easy to see that 
	\begin{equation}\label{(u-k)^+1}
	\int_{A_{R_1}^{R_2}}L(x)|\nabla(\varphi_n-k)^+-\nabla(u-k)^+|^p\diff x \to 0 \quad \text{as}\quad n\to\infty.
	\end{equation}For each $n\in\mathbb{N},$ set $\psi_n:=(\varphi_n-k)^+.$ Fix $n$ and let $R_1<r_1<r_2<R_2$ such that $\operatorname{supp}(\psi_n)\subset A_{r_1}^{r_2}.$ For each $i\in \mathbb{N}$, define $\eta_i(x):=i^N\eta(ix),$ where $\eta$ is a standard normalized mollifier in $\mathbb{R}^N$ and define
	$$v_i^{(n)}(x):=(\eta_i\ast\psi_n)(x)=\int_{\mathbb{R}^N}\eta_i(x-y)\psi_n(y)\diff y.$$
Thus, $v_i^{(n)}\in C^\infty(\mathbb{R}^N)$ for all $i$ and  	$\operatorname{supp}(v_i^{(n)})\subset A_{r_1}^{r_2}$ for  $i$ large. From this together with $L\in L^1(A_{r_1}^{r_2})$ and properties of mollifiers, we obtain
\begin{equation*}
\int_{A_{R_1}^{R_2}}L(x)|\nabla v_i^{(n)}-\nabla\psi_n|^p\diff x \to 0 \quad \text{as}\quad i\to\infty.
\end{equation*}
Thus, we find $i_n$ such that
\begin{equation*}
\int_{A_{R_1}^{R_2}}L(x)|\nabla v_{i_n}^{(n)}-\nabla\psi_n|^p\diff x <\frac{1}{n}\quad \text{i.e.,}\quad  \int_{A_{R_1}^{R_2}}L(x)|\nabla u_n-\nabla(\varphi_n-k)^+|^p\diff x <\frac{1}{n},
\end{equation*}
where $u_n:=v_{i_n}^{(n)}$ $(\in C_c^1(A_{R_1}^{R_2})).$ From here and \eqref{(u-k)^+1}, for such a sequence $\{u_n\}$ we obtain \eqref{(u-k)^+} and the proof is complete.
\end{proof}
\begin{remark}\label{domains}\rm
	Obviously, in this section we can allow $R_1=0$, that is, $A_{R_1}^{R_2}$ is of the form $B_R\setminus\{0\}$ ($0<R\leq\infty$). When $1<p<N$ and $L\in L^1_{\loc}(B_R)$ 
	such that $\lim_{r\to 0}\frac{1}{|B_r|}\int_{B_r}L(x)\diff x<\infty,$ then the space $\mathcal{D}_{0}^{1,p}(A_{0}^{R};L)$ coincides with $\mathcal{D}_{0}^{1,p}(B_R;L),$ the completion of $C_c^1(B_R)$ with respect to the norm
	$$\|u\|=\left(\int_{B_R}L(x)|\nabla u|^p\diff x\right)^{1/p}.$$ 
	That is, $\mathcal{D}_{0}^{1,p}(A_{0}^{R};L)$ is the usual solution space for the Dirichlet problem in a ball $B_R.$
\end{remark}


\section{The eigenvalue problem involving the weighted \texorpdfstring{$p$}{Lg}-Laplacian}\label{Sec.EigenvalueProblem}
In this section we discuss the existence and properties of the first eigenpair of the eigenvalue problem \eqref{1.1}. If $\mathrm{(A)}$ holds and $L^{-s} \in L^1_{\loc}(A_{R_1}^{R_2})$ for some $s\in (\frac{N}{p},\infty)\cap [\frac{1}{p-1},\infty),$ then by the compact embedding $\mathcal{D}_{0}^{1,p}(A_{R_1}^{R_2};L)\hookrightarrow\hookrightarrow L^p(A_{R_1}^{R_2};w)$ and Proposition~\ref{prop.(u-k)^+}, arguing as in \cite[Proof of Lemma 4.1]{Anoop.CV}, we obtain the existence of a principal eigenvalue as follows.
\begin{lemma}\label{3.existence}
Assume that $\mathrm{(A)}$ holds and $L^{-s} \in L^1_{\loc}(A_{R_1}^{R_2})$ for some $s\in (\frac{N}{p},\infty)\cap [\frac{1}{p-1},\infty).$ Then $\lambda_1$ defined in \eqref{def.lamda1} is positive, it is achieved at some $\varphi_1\ge 0$ and $(\lambda_1,\varphi_1)$ is an eigenpair of \eqref{1.1}.
\end{lemma}
The positivity of $\varphi_1$ and the simplicity of $\lambda_1$ can be obtained in the same fashion as in \cite{Kawohl} with suitable modifications. However, the presence of the weight $L$ in the main operator somehow makes the conclusions not to follow in a straightforward manner. For the reader's convenience, we 
sketch the proofs briefly. Note that under the assumption of Theorem~\ref{Theore.eigenpair} we have $u\in W_{\loc}^{1,p_s}(A_{R_1}^{R_2})$ for any (weak) solution $u$ to problem \eqref{1.1} in view of Corollary~\ref{2.corollary.local_embeddings}. 
In fact, we work with the following representation of $u$, defined in $A_{R_1}^{R_2}$ by 
$$u^\ast(x):=\begin{cases}
\lim_{r\to 0}\frac{1}{|B(x,r)|}\int_{B(x,r)}u(y)\diff y & \text{if this limit exists},\\
0 & \text{otherwise}.
\end{cases}$$
In the next lemma, we state a strong maximum principle type result, which is similar to \cite[Proposition 3.2]{Kawohl}. 
\begin{lemma}\label{3.le.smp}
	Assume that $\mathrm{(A)}$ holds and $L^{-s} \in L^1_{\loc}(A_{R_1}^{R_2})$ for some $s\in (\frac{N}{p},\infty)\cap [\frac{1}{p-1},\infty).$ Let $V\in L_{\loc}^1 (A_{R_1}^{R_2})$ and $V\ge 0$. If a nontrivial nonnegative function $u\in \mathcal{D}_{0}^{1,p}(A_{R_1}^{R_2};L)$ satisfies $Vu^{p}\in L_{\loc}^1 (A_{R_1}^{R_2})$ and 
	\begin{equation}\label{3.smp.ineq}
	\int_{A_{R_1}^{R_2}}\big\{L(x)|\nabla u|^{p-2}\nabla u\cdot \nabla\xi +Vu^{p-1}\xi\big\}\diff x \geq 0,\quad \forall \xi\in C_c^\infty(A_{R_1}^{R_2}), \xi\geq 0,
	\end{equation}
	then $\operatorname{Cap}_{p_s}(\mathcal{Z})=0,$ where $\mathcal{Z}:=\big\{x\in A_{R_1}^{R_2}: u(x)=0\big\}.$
\end{lemma}
For the definition of the $p$-capacity $\operatorname{Cap}_p(\cdot)$ and related properties we refer to the book of Evans-Gariepy \cite{Evans} (see also \cite{Kawohl}). 
\begin{proof}We proceed as in \cite[Proof of Proposition 3.2]{Kawohl}. It is worth mentioning that in \cite{Kawohl}, the domain is required to be bounded when $N\leq p$. For each $n\in\mathbb{N},$ denote $\Omega_n:=A_{R_1}^{R_1+n}$ when $R_2=\infty$ and $\Omega_n:=A_{R_1}^{R_2}$  when  $R_2<\infty$ and define $\mathcal{Z}_n:=\big\{x\in \Omega_n: u(x)=0\big\}$. Since $\mathcal{Z}=\bigcup_{n=1}^\infty \mathcal{Z}_n,$ it suffices to show that $\operatorname{Cap}_{p_s}(\mathcal{Z}_n)=0$ for all $n\in\mathbb{N}.$ Let $n$ be fixed. As in\cite[Proof of Proposition 3.2]{Kawohl}, we will show for any $\xi\in C_c^\infty(\Omega_n)$ with $0\leq \xi\le 1$ there exits $C_0=C_0(u,\xi)>0$ such that
	\begin{equation}\label{3.log}
	\int_{\Omega_n}\left|\nabla \log \left(1+\frac{u}{\delta}\right)\right|^{p_s}\xi^{p_s}\diff x\leq C_0,\quad \forall \delta>0.
	\end{equation}
To obtain \eqref{3.log} we use the following identity
\begin{align*}
\int_{\Omega_n}L(x)&\left|\nabla \log \left(1+\frac{u}{\delta}\right)\right|^{p}\xi^{p}\diff x\\
&=\frac{1}{1-p}\int_{\Omega_n}L(x)|\nabla u|^{p-2}\nabla u\cdot\bigg[\nabla \left(\frac{\xi^p}{(u+\delta)^{p-1}}\right)
-p\xi^{p-1}(\nabla\xi)(u+\delta)^{1-p}\bigg]\diff x.
\end{align*}
Then, we use the same argument as in \cite[Proof of Proposition 3.2]{Kawohl}, and employing \eqref{3.smp.ineq}, to obtain
\begin{equation*}
\int_{\Omega_n}L(x)\left|\nabla \log \left(1+\frac{u}{\delta}\right)\right|^{p}\xi^{p}\diff x\\
\leq\int_{\Omega_n}V(x)(1+|u|^p)\xi^{p}\diff x+p^{p-1}\int_{\Omega_n}L(x)|\nabla\xi|^{p}\diff x.
\end{equation*}
Combining this and the estimate
\begin{align*}
\int_{\Omega_n}\bigg|\nabla \log \bigg(1+&\frac{u}{\delta}\bigg)\bigg|^{p_s}\xi^{p_s}\diff x\\
&\leq \left(\int_{\operatorname{supp}(\xi)}L^{-s}(x)\diff x\right)^{\frac{1}{s+1}}\left(\int_{\Omega_n}L(x)\left|\nabla \log \left(1+\frac{u}{\delta}\right)\right|^{p}\xi^{p}\diff x\right)^{\frac{s}{s+1}},
\end{align*}
we obtain \eqref{3.log}. The rest of the proof is similar to that of \cite[Proof of Proposition 3.2]{Kawohl}.
\end{proof}
Finally, we sketch the proof of Theorem~\ref{Theore.eigenpair}.
\begin{proof}[Proof of Theorem~\ref{Theore.eigenpair}]
By Lemma~\ref{3.existence}, we have $\lambda_1$ is a positive eigenvalue of \eqref{1.1} and there is a nonnegative eigenfunction $\varphi_1$ associated with $\lambda_1.$ Since
\begin{equation*}
\int_{A_{R_1}^{R_2}}\big\{L(x)|\nabla \varphi_1|^{p-2}\varphi_1\cdot \nabla\xi +\lambda_1K^-\varphi_1^{p-1}\xi\big\}\diff x= \lambda_1\int_{A_{R_1}^{R_2}}K^+\varphi_1^{p-1}\xi\diff x\geq 0
\end{equation*}
for all $\xi\in C_c^\infty(A_{R_1}^{R_2}), \xi\geq 0,$ we get $\varphi_1>0$ a.e. in $A_{R_1}^{R_2}$ in view of Lemma~\ref{3.le.smp}. The simplicity of $\lambda_1$ can be proved by the same argument as \cite[Proof of Theorem 1.3]{Kawohl} for which we invoke Lemma~\ref{3.le.smp} and use $p_s$-capacity instead of $p$-capacity.
\end{proof}
\begin{remark}\label{problems_on_balls}\rm
	Similarly to Section~\ref{Sec.Preliminaries}, in this section we can also allow $R_1=0$. As shown in Remark~\ref{domains}, when $1<p<N$ and $L\in L^1_{\loc}(B_R)$ such that $\lim_{r\to 0}\frac{1}{|B_r|}\int_{B_r}L(x)\diff x<\infty$ also in this section we recover results for a ball $B_R$ ($0<R\leq\infty$).
\end{remark}

\section{Qualitative properties of solutions}\label{Sec.Boundedness}
In this section we prove qualitative properties of solutions mentioned in 
Section~\ref{introduction} (Theorems~\ref{Theorem.local.behaviorI}--
\ref{Theorem.everywhere-positivity} and Corollaries~
\ref{Coro.decay.at.infinity.2}--\ref{Coro.decay.at.R1}). 
\subsection{Boundedness of solutions} In this subsection, we obtain the (local) boundedness of solutions to problem~\eqref{1.1}. As we mentioned in Section~\ref{introduction}, the boundedness of solutions can be obtained for more general nonlinear term via the De Giorgi type iterations technique. More precisely, consider the following problem
\begin{equation}\label{P}
-\operatorname{div}\left( L(x) |\nabla u| ^{p-2}\nabla u\right)=f(x,u) \quad \text{a.e.\ in }A_{R_1}^{R_2} ,
\end{equation}
where the weight $L$ satisfies the condition $\mathrm{(A1)}$ in the Corollary~\ref{2.corollary.local_embeddings} and the nonlinear term $f$ satisfies
\begin{itemize}
	\item[(F)] $f: A_{R_1}^{R_2}\times \mathbb{R} \to \mathbb{R}$ is a Carath\'{e}odory function such that $|f(x,\tau)| \leq a(x)|\tau|^{p-1} +b(x)$ for a.e. $x \in A_{R_1}^{R_2}$ and all $\tau\in \mathbb{R}$, where $a,b$ are nonnegative measurable functions in $A_{R_1}^{R_2}.$  
\end{itemize}
\begin{definition}\rm
	By a weak solution of problem \eqref{P}, we mean a function $u\in \mathcal{D}_{0}^{1,p}(A_{R_1}^{R_2};L)$ such that $f(\cdot,u)\in L_{\loc}^1(A_{R_1}^{R_2})$ and 
	\begin{equation*}     
		\int_{A_{R_1}^{R_2}}L(x)|\nabla u|^{p-2}\nabla u\cdot \nabla \xi \diff x= \int_{A_{R_1}^{R_2}}f(x,u)\xi\diff x, \quad \forall \xi \in  C_c^1(A_{R_1}^{R_2}).   \label{4defweaksol}
	\end{equation*}
\end{definition}
\begin{theorem}\label{apriorimainthm}
		Assume that $\mathrm{(A1)}$ and $\mathrm{(F)}$ hold. 
	\begin{itemize}
		\item[(i)] Assume in addition that 
		$L,a \in L^{\frac{q}{q-p}}(A_{R_1}^{R_1+2\epsilon}),$ $b \in L^{\frac{t}{t-1}} (A_{R_1}^{R_1+2\epsilon})$ and $L^{-s} \in L^1(A_{R_1}^{R_1+2\epsilon})$ for some $\epsilon\in (0,\frac{R_2-R_1}{2}),$ $s\in (\frac{N}{p},\infty)\cap [\frac{1}{p-1},\infty),$ $q\in [p,p_s^{\ast})$ and $t\in [1,q]\cap [1,\frac{p_s^{\ast}}{p}).$  Then for any weak solution $u$ of problem \eqref{P},
		we have $u\in L^q(A_{R_1}^{R_1+2\epsilon})\cap L^{\infty}(A_{R_1}^{R_1+\epsilon})$ and 
		\begin{equation} \label{mainestI}
		\|u\|_{L^\infty\big(A_{R_1}^{R_1+\epsilon}\big)}\leq C\left[1+\left(\int_{A_{R_1}^{R_1+2\epsilon}}|u|^q\diff x\right)^{\mu}\right],
		\end{equation}
		where $C,\mu>0$ are independent of $u$.
		\item[(ii)] Assume in addition that $L,a \in L^{\frac{q}{q-p}}(B(x_0,r_0)),$ $b \in L^{\frac{t}{t-1}} (B(x_0,r_0))$ and $L^{-s} \in L^1(B(x_0,r_0))$ for some ball $B(x_0,r_0)\subset A_{R_1}^{R_2},$ $s\in (\frac{N}{p},\infty)\cap [\frac{1}{p-1},\infty),$ $q\in [p,p_s^{\ast})$ and $t\in [1,q]\cap [1,\frac{p_s^{\ast}}{p}).$
		Then for any weak solution $u$ of problem \eqref{P},
		we have $u\in L^q(B(x_0,r_0))\cap L^{\infty}(B(x_0,\frac{r_0}{2}))$ and 
		\begin{equation*} 
		\|u\|_{L^\infty\big(B\big(x_0,\frac{r_0}{2}\big)\big)}\leq C\left[1+\left(\int_{B(x_0,r_0)}|u|^q\diff x\right)^{\mu}\right],
		\end{equation*}
		where $C,\mu>0$ are independent of $u$.  In particular, if $L,a \in L_{\loc}^{\frac{q}{q-p}}(A_{R_1}^{R_2}),$ $b \in L_{\loc}^{\frac{t}{t-1}} (A_{R_1}^{R_2})$ and $L^{-s} \in L_{\loc}^1(A_{R_1}^{R_2})$ then $u\in L^\infty_{\loc}(A_{R_1}^{R_2}).$
	\end{itemize}
\end{theorem}
To prove Theorem~\ref{apriorimainthm} we first prove the following lemma.
\begin{lemma}\label{well-defined}
		Assume that $\mathrm{(A1)}$ holds.
		\begin{itemize}
			\item[(i)]  If $L^{-s} \in L^1(A_{R_1}^{R_1+2\epsilon})$ for some $\epsilon\in (0,\frac{R_2-R_1}{2}),$ then $\mathcal{D}_0^{1,p}(A_{R_1}^{R_2};L)\hookrightarrow W^{1,p_s}(A_{R_1}^{R_1+2\epsilon})$ and hence $\mathcal{D}_0^{1,p}(A_{R_1}^{R_2};L)\hookrightarrow L^q(A_{R_1}^{R_1+2\epsilon})$ for $q\in [1,p_s^\ast).$
			\item[(ii)]  If $L^{-s} \in L^1(B(x_0,r_0))$ for some ball $B(x_0,r_0)\subset A_{R_1}^{R_2},$ then $\mathcal{D}_0^{1,p}(A_{R_1}^{R_2};L)\hookrightarrow W^{1,p_s}(B(x_0,r_0))$ and hence $\mathcal{D}_0^{1,p}(A_{R_1}^{R_2};L)\hookrightarrow L^q(B(x_0,r_0))$ for $q\in [1,p_s^\ast).$
		\end{itemize}
\end{lemma}
\begin{proof}
(i) Let $u\in \mathcal{D}_0^{1,p}(A_{R_1}^{R_2};L)$ and let $\{u_n\}\subset C_c^1(A_{R_1}^{R_2})$ such that $u_n\to u$ in $\mathcal{D}_0^{1,p}(A_{R_1}^{R_2};L)$ as $n\to\infty.$ 
By Corollary~\ref{2.corollary.local_embeddings} (i), up to a subsequence we have $u_n\to u$ a.e. in $A_{R_1}^{R_2}.$ Let $\phi\in C^\infty(\mathbb{R}^N)$ such that $\chi_{B_{R_1+\epsilon}}\leq \phi\leq \chi_{B_{R_1+\frac{3\epsilon}{2}}},$ where $\chi_\Omega$ denotes the characteristic function on the set $\Omega.$ Then $\phi u_n\in C_c^1(A_{R_1}^{R_1+2\epsilon}).$ Thus, by Poincar\'e's inequality there exists a positive constant $C$ such that
$$\int_{A_{R_1}^{R_1+2\epsilon}}|\phi u_n|^{p_s}\diff x\leq C \int_{A_{R_1}^{R_1+2\epsilon}}|\nabla(\phi u_n)|^{p_s}\diff x,\quad \forall n\in\mathbb{N}.$$
Hence, applying H\"older's inequality and the embedding $\mathcal{D}_0^{1,p}(A_{R_1}^{R_2};L)\hookrightarrow L^p(A_{R_1+\epsilon}^{R_1+\frac{3\epsilon}{2}})$ we obtain from the last inequality that
\begin{align*}
&\int_{A_{R_1}^{R_1+\epsilon}}|u_n|^{p_s}\diff x\leq C_1 \int_{A_{R_1}^{R_1+2\epsilon}}|\nabla u_n|^{p_s}\diff x+ C_1\int_{A_{R_1+\epsilon}^{R_1+\frac{3\epsilon}{2}}}|u_n|^{p_s}\diff x\\
&\leq C_1 \left(\int_{A_{R_1}^{R_1+2\epsilon}}L^{-s}(x)\diff x\right)^{\frac{1}{s+1}}\left(\int_{A_{R_1}^{R_1+2\epsilon}}L(x)|\nabla u_n|^p\diff x\right)^{\frac{s}{s+1}}+ C_2\left(\int_{A_{R_1+\epsilon}^{R_1+\frac{3\epsilon}{2}}}|u_n|^p\diff x\right)^{\frac{s}{s+1}}\\
&\leq C_3 \left(\int_{A_{R_1}^{R_2}}L(x)|\nabla u_n|^p\diff x\right)^{\frac{s}{s+1}},\quad \forall n\in\mathbb{N}.
\end{align*}
Letting $n\to\infty$ and invoking Fatou's lemma we obtain the above estimate for $u_n=u$.  Combining this with the embedding $\mathcal{D}_0^{1,p}(A_{R_1}^{R_2};L)\hookrightarrow L^p(A_{R_1+\epsilon}^{R_1+2\epsilon})$ and the estimate
 $$\int_{A_{R_1}^{R_1+2\epsilon}}|\nabla u|^{p_s}\diff x
\leq \left(\int_{A_{R_1}^{R_1+2\epsilon}}L^{-s}(x)\diff x\right)^{\frac{1}{s+1}}\left(\int_{A_{R_1}^{R_1+2\epsilon}}L(x)|\nabla u|^p\diff x\right)^{\frac{s}{s+1}},$$
we deduce $\|u\|_{ W^{1,p_s}(A_{R_1}^{R_1+2\epsilon})}\leq C_4\|u\|$ for some constant $C_4$ independent of $u$.

(ii) The conclusion is clear in view of \cite[p. 25, the embedding (1.22)]{Dra-Kuf-Nic}. 
\end{proof}

To employ the De Giorgi iteration, we need the following key lemma. The special case $\delta_1=\delta_2$ was obtained in \cite[Ch.2, lemma 4.7] {Ladyzhenskaya}.  

\begin{lemma}(\cite[Lemma 4.3]{Ho-Sim})\label{leRecur}
	Let $\{J_n\}_{n=0}^{\infty}$  be a sequence  of positive numbers satisfying the recursion inequality
	\begin{equation}\label{Recur.Int}
		J_{n+1}\leq K \eta^n \left( J_{n}^{1+\delta_1}+J_{n}^{1+\delta_2}\right),\quad n=0,1,2,\cdots,
	\end{equation}
	for some $\eta>1, \ K>0 \ and\ \delta_2 \geq \delta_1 >0$. If $J_{0}\leq \min\left(1,(2K)^{\frac{-1}{\delta_1}}\ \eta^{\frac{-1}{\delta_1^{2}}}\right) $ or
	\begin{equation*}\label{Y0.Le}
		J_{0}\leq \min\left((2K)^{\frac{-1}{\delta_1}}\ \eta^{\frac{-1}{\delta_1^{2}}},(2K)^{\frac{-1}{\delta_2}}\ \eta^{-\frac{1}{\delta_1\delta_2}-\frac{\delta_2-\delta_1}{\delta_2^{2}}}\right),
	\end{equation*}
	then there exists $n \in \mathbb{N}\cup \{0\}=:\mathbb{N}_0$ such that $J_n\leq 1$. Moreover,
	$$
	J_{n}\leq \min\left(1,(2K)^{\frac{-1}{\delta_1}}\ \eta^{\frac{-1}{\delta_1^{2}}}\ \eta^{\frac{-n}{\delta_1}}\right),\ \forall n\geq n_0,
	$$
	where $n_0$ is the smallest $n\in \mathbb{N}_0$ for which $J_n\leq 1$. In particular, $J_n \to 0$ as $n \to \infty$.
\end{lemma}

\begin{proof}[Proof of Theorem~\ref{apriorimainthm}]
(i) Let $u$ be a weak solution of problem \eqref{4defweaksol}.  In the rest of the proof of the theorem, the constant $C$ might vary from line to line, but will be always independent of $L,\ a,\ b,\ \epsilon$ and $u$. Without loss of generality we may assume that $t>\frac{q}{p}.$

\textbf{Step 1: Caccioppoli-type inequality}. Denote
\begin{equation}\label{4.notations}
 \alpha:= \|L\|_{L^{\frac{q}{q-p}}\big(A^{R_1+2\epsilon}_{R_1}\big)}, \ \beta:= \|a\|_{L^{\frac{q}{q-p}}\big(A^{R_1+2\epsilon}_{R_1}\big)}  \text{ and } \gamma:= \|b\|_{L^{\frac{t}{t-1}}\big(A^{R_1+2\epsilon}_{R_1}\big)}, 
\end{equation}
and for $k>0, \ r \in (R_1,R_2)$, denote
$$A_{k,r}:= \{x \in A^{r}_{R_1}: u(x)>k\}.$$ 
We claim that there exists a positive constant $C$ such that, for any $r_1,\ r_2$ satisfying  $R_1+\epsilon\leq r_1<r_2\leq R_1+2\epsilon$ and for any $k>0$ we have
\begin{align}
\int_{A_{k,r_1}}L(x) |\nabla u|^p \diff x &\leq C(\alpha +\beta\epsilon^p) \bigg(\int_{A_{k,r_2}} \bigg(\frac{u-k}{r_2-r_1} \bigg)^q\diff x \bigg)^{\frac{p}{q}} + \notag \\&\quad + p\gamma \bigg( \int_{A_{k,r_2}} (u-k)^q\diff x \bigg)^{\frac{1}{q}} |A_{k,r_2}|^{\frac{q-t}{qt}} +C\beta k^p |A_{k,r_2}|^{\frac{p}{q}}.    \label{Caccioppoli}
\end{align}
To this end, let $\xi \in C^1(\mathbb{R}^N)$ such that $$\chi_{B_{r_1}} \leq \xi \leq \chi_{B_{r_2}} \text { and } |\nabla \xi | \leq \frac{2}{r_2-r_1}.$$ By an approximation argument, we can show that for $\widetilde{u}\in \mathcal{D}_0^{1,p}(A_{R_1}^{R_2};L)$ and $\widetilde{\xi}\in C^1(\mathbb{R}^N)$ with $\chi_{B_{r_1}} \leq \widetilde{\xi} \leq \chi_{B_{r_2}},$ we have $\widetilde{u}\widetilde{\xi} \in \mathcal{D}_0^{1,p}(A_{R_1}^{R_2};L)$ and $\widetilde{u}\widetilde{\xi}$ is a test function for \eqref{P}. By this and Proposition~\ref{prop.(u-k)^+}, we can use $(u-k)^+ \xi^p$ as a test function in \eqref{P} and get
	\begin{equation*}
	\int_{A_{R_1}^{R_2}} L(x) |\nabla u|^{p-2} \nabla u \cdot \nabla((u-k)^+\xi^p) \diff x = \int_{A_{R_1}^{R_2}}  f(x,u)(u-k)^+\xi^p \diff x.
	\end{equation*}
By the assumption on $f$, the last equality leads to 
	\begin{align*}
	\int_{A_{k,r_2}} L(x)|\nabla u|^{p} \xi^p \diff x  \leq &-p\int_{A_{k,r_2}}L(x) |\nabla u|^{p-2} (\nabla u \cdot \nabla \xi) (u-k)\xi^{p-1} \diff x \\  
	&+ \int_{A_{k,r_2}} a(x)|u|^{p-1} (u-k) \xi^p \diff x  + \int_{A_{k,r_2}} b(x)(u-k)\xi^p \diff x.
	\end{align*}
That is 
	\begin{align}
	\int_{A_{k,r_2}} L(x) |\nabla u|^p \xi^p \diff x &\leq p \int_{A_{k,r_2}} L(x) |\nabla u|^{p-1} \xi^{p-1} |\nabla \xi| (u-k) \diff x \notag\\
	&\quad + \int_{A_{k,r_2}} a(x)u^p \diff x + \int_{A_{k,r_2}}  b(x)(u-k) \diff x.	    \label{aprioriest1}	\end{align}
Now we estimate three integrals on the right hand side (RHS for short) of \eqref{aprioriest1} separately. For simplicity, denote
\[ J:=\int_{A_{k,r_2}} L(x) |\nabla u|^p \xi^p \diff x \text{ and } Q:= \int_{A_{k,r_2}} \bigg(\frac{u-k}{r_2-r_1} \bigg)^q\diff x .\]
We estimate the first integral on RHS of \eqref{aprioriest1}, using Young's inequality and H\"older's inequality, as follows 
	\begin{align}
	\int_{A_{k,r_2}}& L(x) |\nabla u|^{p-1} \xi^{p-1} |\nabla \xi| (u-k) \diff x\notag\\
	& \leq \frac{p-1}{p}\int_{A_{k,r_2}} L(x)\frac{1}{p} |\nabla u|^p \xi^p \diff x  +  \frac{1}{p} \int_{A_{k,r_2}} L(x)p^{p-1}  (|\nabla \xi|(u-k))^p \diff x   \notag \\
	&\leq \frac{p-1}{p^2} J + 2^p p^{p-2} \int_{A_{k,r_2}} L(x) \bigg(\frac{u-k}{r_2-r_1}\bigg)^p \diff x    \notag \\
	&\leq \frac{p-1}{p^2} J +2^p p^{p-2} \|L\|_{L^{\frac{q}{q-p}}\big(A^{R_1+2\epsilon}_{R_1}\big)} 
	\bigg(\int_{A_{k,r_2}} \bigg( \frac{u-k}{r_2-r_1}\bigg)^q \diff x\bigg)^{\frac{p}{q}}   \notag \\
	&= \frac{p-1}{p^2} J + 2^p p^{p-2} \alpha Q^{\frac{p}{q}}.   \label{aprioriiest2}
	\end{align}
	Using H\"older's inequality, we estimate the second integral on RHS of \eqref{aprioriest1}
	\begin{align}
	\int_{A_{k,r_2}} a(x) u^p \diff x &\leq \|a\|_{L^{\frac{q}{q-p}}\big(A^{R_1+2\epsilon}_{R_1}\big)}  \bigg(\int_{A_{k,r_2}} u^q \diff x \bigg)^{\frac{p}{q}}   \notag \\ 
	&\leq \beta  \bigg[\int_{A_{k,r_2}} 2^q\big((u-k)^q+k^q\big) \diff x \bigg]^{\frac{p}{q}}   \notag \\ 
	&\leq C \beta \epsilon^pQ^{\frac{p}{q}} + C \beta  k^p |A_{k,r_2}|^{\frac{p}{q}}.     \label{aprioriest3}
	\end{align} 
	Using H\"older's inequality again, we estimate the third integral on RHS of \eqref{aprioriest1}
	\begin{align}
	\int_{A_{k,r_2}} b(x)(u-k) \diff x &\leq \|b\|_{L^{\frac{t}{t-1}}\big(A^{R_1+2\epsilon}_{R_1}\big)} \bigg( \int_{A_{k,r_2}} (u-k)^t \diff x\bigg)^{\frac{1}{t}}    \notag\\
	&\leq \gamma \bigg( \int_{A_{k,r_2}} (u-k)^q \diff x\bigg)^{\frac{1}{q}}  |A_{k,r_2}|^{\frac{q-t}{qt}}.   \label{aprioriest4}
	\end{align}
	From \eqref{aprioriest1}--\eqref{aprioriest4}, we obtain
	$$J \leq \frac{p-1}{p} J + 2^p p^{p-1} \alpha Q^{\frac{p}{q}} + C \beta\epsilon^p Q^{\frac{p}{q}}+ C \beta  k^p |A_{k,r_2}|^{\frac{p}{q}}+ \gamma \bigg( \int_{A_{k,r_2}} (u-k)^q \diff x\bigg)^{\frac{1}{q}}  |A_{k,r_2}|^{\frac{q-t}{qt}}.$$
	Hence 
	\begin{align*}
	J \leq C(\alpha + \beta\epsilon^p ) Q^{\frac{p}{q}} +p\gamma \bigg(\int_{A_{k,r_2}}(u-k)^q \diff x\bigg)^{\frac{1}{q}} |A_{k,r_2}|^{\frac{q-t}{qt}} +C \beta k^p |A_{k,r_2}|^{\frac{p}{q}}.
	\end{align*}
From this and the definitions of $J,\ Q$ and $\xi$ we obtain \eqref{Caccioppoli}. 
	
\textbf{Step 2: Definition of recursive sequence and recursion inequality.}	Define the recursive sequence $\{J_n\}$ as
\[J_n := \int_{A_{k_n,\rho_n}} (u-k_n)^q \diff x, \ \ \forall n \in \mathbb{N}_0,\]
where $\rho_n:=R_1 +\epsilon+ \frac{\epsilon}{2^n}$ and $k_n:= k_{\ast} \big(1-\frac{1}{2^{n+1}}\big)$ for some $k_{\ast}>1$, to be specified later. 
We also denote $\bar{\rho}_n:= \frac{\rho_n+\rho_{n+1}}{2}$ ($n\in\mathbb{N}_0$). Clearly, $ \rho_n \downarrow R_1+\epsilon,\ k_n \uparrow k_{\ast},\ R_1+\epsilon < \rho_n \leq R_1+2\epsilon$ and $\frac{k_{\ast}}{2} \leq k_n < k_{\ast}$ for all $n \in \mathbb{N}_0.$ Moreover, notice that \begin{align*}
\rho_n-\bar{\rho}_n= \frac{\epsilon}{2^{n+2}},\  k_{n+1}-k_n = \frac{k_{\ast}}{2^{n+2}}, \ \ \forall n \in \mathbb{N}_0.
\end{align*}
Next, we obtain a recursion inequality of the form \eqref{Recur.Int}. Fix $\zeta \in C^1(\mathbb{R})$, such that $\chi_{(-\infty,1)} \leq \zeta \leq \chi_{\big(-\infty,\frac{3}{2}\big)}$ and $|\zeta'| \leq 4$. Define 
\[\zeta_n(x)= \zeta\bigg( \frac{2^{n+1}}{\epsilon}(|x|-R_1-\epsilon)\bigg), \ \ \forall n \in \mathbb{N}_0.\]
Thus, $\zeta_n\in C^1(\mathbb{R}^N)$ and satisfies
\begin{equation*}
\chi_{B_{\rho_{n+1}}}\leq \zeta_n\leq \chi_{B_{\bar{\rho}_n}} \text{ and }|\nabla \zeta_n | \leq \frac{2^{n+3}}{ \epsilon}, \ \ \forall n \in \mathbb{N}_0.   
\end{equation*}
Before estimating $J_{n+1}$ in terms of $J_n$ we note that 
\begin{align}
\int_{A_{k_{n+1},\bar{\rho}_{n}}} (u-k_{n+1})^q \diff x  \leq \int_{A_{k_{n+1},\rho_n}} (u-k_{n+1})^q \diff x \leq J_n,      \label{aprioriest13}
\end{align}
also 
\begin{align}
\big|A_{k_{n+1},{\rho}_{n+1}}\big| \leq \big|A_{k_{n+1},\bar{\rho}_{n}}\big| \leq \big|A_{k_{n+1},{\rho}_{n}}\big| &\leq \int_{A_{k_{n+1},\rho_n}} \bigg(\frac{u-k_n}{k_{n+1}-k_n}\bigg)^q \diff x \leq 2^{(n+2)q}k_{\ast}^{-q}J_n.    \label{aprioriest14}
\end{align}
Furthermore, we will need the following simple inequality
\begin{equation}\label{simple.inq}
(x+y)^m\leq C_m(x^m+y^m), \quad \forall x, y\geq 0\quad (m\geq 0).
\end{equation}
Now, fix  $\bar{q} \in (tp,p_s^{\ast})$. Using H\"older's inequality we estimate
\begin{equation} 
J_{n+1} = \int_{A_{k_{n+1},\rho_{n+1}}} (u-k_{n+1})^q \diff x \leq \bigg(\int_{A_{k_{n+1},\rho_{n+1}}} (u-k_{n+1})^{\bar{q}} \diff x\bigg)^{\frac{q}{\bar{q}}}\bigg|A_{k_{n+1},\rho_{n+1}}\bigg| ^{\frac{\bar{q}-q}{\bar{q}}}.  \label{aprioriest7}
\end{equation}
On the other hand, in view of Lemma~\ref{well-defined} and Sobolev's embedding, we get 
\begin{align}
\bigg(\int_{A_{k_{n+1},\rho_{n+1}}} &(u-k_{n+1})^{\bar{q}} \diff x\bigg)^{\frac{1}{\bar{q}}} = \bigg(\int_{A_{k_{n+1},\rho_{n+1}}}\big((u-k_{n+1})\zeta_n\big)^{\bar{q}} \diff x\bigg)^{\frac{1}{\bar{q}}}   \notag\\
&\leq \bigg(\int_{A_{R_1}^{R_1+2\epsilon}} \big((u-k_{n+1})^+\zeta_n\big)^{\bar{q}} \diff x\bigg)^{\frac{1}{\bar{q}}}   \notag\\
&\leq C_{\epsilon} \bigg[ \bigg(\int_{A_{R_1}^{R_1+2\epsilon}} \big((u-k_{n+1})^+\zeta_n\big)^{p_s} \diff x\bigg)^{\displaystyle\frac{1}{p_s}} +  \notag \\ &\qquad \qquad \qquad+\bigg(\int_{A_{R_1}^{R_1+2\epsilon}} |\nabla\big((u-k_{n+1})^+\zeta_n)|^{p_s} \diff x\bigg)^{\displaystyle\frac{1}{p_s}} \bigg],     \label{aprioriest8}
\end{align}
here $C_{\epsilon}$ is the embedding constant for $W^{1,p_s}(A_{R_1}^{R_1+2\epsilon}) \hookrightarrow L^{\bar{q}}(A_{R_1}^{R_1+2\epsilon})$. Using H\"older's inequality, we have
\begin{align*}
\int_{A_{R_1}^{R_1+2\epsilon}} \big((u-k_{n+1})^+\zeta_n\big)^{p_s} \diff x &\leq \int_{A_{k_{n+1},\bar{\rho}_{n}}} \displaystyle (u-k_{n+1})^{p_s} \diff x      \\
&\leq  \bigg(\int_{A_{k_{n+1},\bar{\rho}_{n}}} (u-k_{n+1})^q\diff x \bigg)^{\displaystyle\frac{p_s}{q}} \big|A_{k_{n+1},\bar{\rho}_{n}}\big|^{\displaystyle\frac{q- p_s}{q}}.
\end{align*}
Combining this with \eqref{aprioriest13} and \eqref{aprioriest14} we obtain
\begin{equation}
\bigg(\int_{A_{R_1}^{R_1+2\epsilon}} \big((u-k_{n+1})^+\zeta_n\big)^{p_s} \diff x\bigg)^{\displaystyle \frac{1}{p_s}} \leq 2^{\displaystyle\frac{2(q-p_s)}{p_s}} 2^{\displaystyle\frac{n(q-p_s)}{p_s}} k_\ast^{-\displaystyle\frac{q-p_s}{p_s}}J_n^{\displaystyle\frac{1}{p_s}}.    \label{aprioriest9}
\end{equation}
We also have
\begin{align}
\bigg(\int_{A_{R_1}^{R_1+2\epsilon}} &|\nabla\big((u-k_{n+1})^+\zeta_n)|^{p_s} \diff x\bigg)^{\frac{1}{p_s}} \leq \bigg(\int_{A_{R_1}^{R_1+2\epsilon}} L^{-s}(x) \diff x\bigg)^{\frac{1}{sp}} \times \notag\\ &\qquad \qquad \qquad \times\bigg(\int_{A_{R_1}^{R_1+2\epsilon}} L(x) |\nabla\big((u-k_{n+1})^+\zeta_n)|^p \diff x\bigg)^{\frac{1}{p}} \notag\\
&\leq 2\delta \bigg[\int_{A_{k_{n+1},\bar{\rho}_{n}}} L(x) |\nabla u|^p \diff x +  2^{(n+3)p} \epsilon^{-p} \int_{A_{k_{n+1},\bar{\rho}_{n}}} L(x)(u-k_{n+1})^p \diff x \bigg]^{\frac{1}{p}}    \notag\\
&\leq 2\delta  \bigg[ \int_{A_{k_{n+1},\bar{\rho}_{n}}} L(x) |\nabla u|^p \diff x + 2^{(n+3)p}\epsilon^{-p} \alpha \bigg(\int_{A_{k_{n+1},\bar{\rho}_{n}}}(u-k_{n+1})^q \diff x\bigg)^{\frac{p}{q} }\bigg] ^{\frac{1}{p}},    \label{aprioriest10}
\end{align}
where $\delta:=\bigg(\displaystyle\int_{A_{R_1}^{R_1+2\epsilon}} L^{-s}(x) \diff x\bigg)^{\frac{1}{sp}}$ and $\alpha$ is as in \eqref{4.notations}.
From \eqref{aprioriest13} and \eqref{aprioriest8}-\eqref{aprioriest10}, invoking \eqref{simple.inq}, we get 
\begin{align}
\bigg(\int_{A_{k_{n+1},\rho_{n+1}}} (u-k_{n+1})^{\bar{q}} \diff x\bigg)^{\frac{q}{\bar{q}}} \leq CC_{\epsilon}^q \bigg\{ &2^{\frac{n(q-p_s)}{p_s}} k_\ast^{-\frac{q-p_s}{p_s}}J_n^{\frac{1}{p_s}} +2^n\epsilon^{-1} \alpha^{\frac{1}{p}} \delta J_n^{\frac{1}{q}}\notag \\ &+\delta  \bigg( \int_{A_{k_{n+1},\bar{\rho}_{n}}} L(x) |\nabla u|^p \diff x \bigg) ^{\frac{1}{p}}\bigg\}^q . \label{aprioriest11}
\end{align}
Applying \eqref{Caccioppoli} with $r_1 =\bar{\rho}_n,\ r_2=\rho_n$ and $k=k_{n+1}$, we get
\begin{align*}
\int_{A_{k_{n+1},\bar{\rho}_{n}}} L(x)&|\nabla u|^p \diff x \leq C(\alpha+\beta\epsilon^p) \epsilon^{-p} 2^{np} \bigg(\int_{A_{k_{n+1},\rho_n}} (u-k_{n+1})^q \diff x\bigg)^{\frac{p}{q}}  + \\ &+ p\gamma \bigg( \int_{A_{k_{n+1},\rho_n}} (u-k_{n+1})^q\diff x \bigg)^{\frac{1}{q}} |A_{k_{n+1},\rho_n}|^{\frac{q-t}{qt}} +C\beta k_{\ast}^p |A_{k_{n+1},\rho_n}|^{\frac{p}{q}}.
\end{align*}
Using \eqref{aprioriest13}  and  \eqref{aprioriest14} again, we deduce from the last inequality that
$$\int_{A_{k_{n+1},\bar{\rho}_{n}}} L(x)|\nabla u|^p \diff x \leq C(\epsilon^{-p}\alpha+\beta)2^{np} J_n^{\frac{p}{q}}  + C\gamma 2^{\frac{n(q-t)}{t}} k_\ast^{-\frac{q-t}{t}}J_n^{\frac{1}{t}}.$$ 
Invoking \eqref{simple.inq} the last inequality yields
$$\left(\int_{A_{k_{n+1},\bar{\rho}_{n}}} L(x)|\nabla u|^p \diff x\right)^{\frac{1}{p}} \leq C(\epsilon^{-1}\alpha^{\frac{1}{p}}+\beta^{\frac{1}{p}})2^{n} J_n^{\frac{1}{q}}  + C\gamma^{\frac{1}{p}} 2^{\frac{n(q-t)}{tp}} k_\ast^{-\frac{q-t}{tp}}J_n^{\frac{1}{tp}}.$$
From this and \eqref{aprioriest11}, we obtain 
\begin{align}
\bigg(\int_{A_{k_{n+1},\rho_{n+1}}} (u-k_{n+1})^{\bar{q}} \diff x\bigg)^{\frac{q}{\bar{q}}} \leq CC_{\epsilon}^q \bigg\{ &2^{\frac{n(q-p_s)}{p_s}} k_\ast^{-\frac{q-p_s}{p_s}}J_n^{\frac{1}{p_s}}\notag +\delta(\epsilon^{-1}\alpha^{\frac{1}{p}}+\beta^{\frac{1}{p}})2^n J_n^{\frac{1}{q}}\\
&+ \gamma^{\frac{1}{p}} 2^{\frac{n(q-t)}{tp}} k_\ast^{-\frac{q-t}{tp}}J_n^{\frac{1}{tp}} \bigg\}^q \label{aprioriest12}.
\end{align}
It follows from \eqref{aprioriest12} and \eqref{simple.inq}, noticing $k^\ast>1$ and $J_n^{\frac{1}{p_s}}+J_n^{\frac{1}{q}}+J_n^{\frac{1}{tp}}\leq 2(J_n^{\frac{1}{p_s}}+J_n^{\frac{1}{tp}})$ due to $p_s<q<tp,$  that
\begin{align}
\bigg(&\int_{A_{k_{n+1},\rho_{n+1}}} (u-k_{n+1})^{\bar{q}} \diff x \bigg)^{\frac{q}{\bar{q}}} \leq \widetilde{C}(\epsilon,\alpha,\beta,\gamma,\delta) 2^{\frac{nq^2}{p_s}} \big(J_n^{\frac{q}{p_s}}+J_n^{\frac{q}{tp}}\big).     \label{aprioriest15}
\end{align}
From \eqref{aprioriest7}, \eqref{aprioriest14} and \eqref{aprioriest15}, we obtain 
\begin{align*}
J_{n+1} \leq C(\epsilon,\alpha,\beta,\gamma,\delta) 2^{\frac{nq^2}{p_s}} \bigg( J_n^{\frac{q}{p_s}} +J_n^{\frac{q}{tp}}\bigg) 2^{\frac{nq(\bar{q}-q)}{\bar{q}}}  k_{\ast}^{-\frac{q(\bar{q}-q)}{\bar{q}}} J_n^{\frac{\bar{q}-q}{\bar{q}}}.
\end{align*}
That is, 
\begin{align}\label{Recur.Inq (i)}
J_{n+1} \leq C(\epsilon,\alpha,\beta,\gamma,\delta)  k_{\ast} ^{-\frac{q(\bar{q}-q)}{\bar{q}}} \eta ^n \bigg(J_n^{1+\delta_1}+J_n^{1+\delta_2}\bigg),
\end{align}
where 
\begin{align*}
0<\delta_1:= \frac{q}{tp} -\frac{q}{\bar{q}}  < \delta_2 := \frac{q}{p_s} - \frac{q}{\bar{q}} \text{ and }
\eta:= 2^{\frac{q^2}{p_s}+\frac{q(\bar{q}-q)}{\bar{q}}} >1.
\end{align*}

\textbf{Step 3: A-priori bounds.} Invoking Lemma~\ref{leRecur}, we deduce from \eqref{Recur.Inq (i)} that  $J_n \to 0$ as $n \to \infty$, provided 
\begin{equation}
J_0 \leq \min \bigg((2\widetilde{k})^{-\frac{1}{\delta_1}} \eta^{-\frac{1}{\delta_1^2}}, (2\widetilde{k})^{-\frac{1}{\delta_2}} \eta^{-\frac{1}{\delta_1\delta_2} -\frac{\delta_2-\delta_1}{\delta_2^2}}\bigg),   \label{{aprioriest16}}
\end{equation}
where $\widetilde{k}:= C(\epsilon,\alpha,\beta,\gamma,\delta) k_{\ast}^{-\frac{q(\bar{q}-q)}{\bar{q}}}$.
We have 
\begin{align*}
J_0= \int_{A_{k_0,\rho_0}}(u-k_0)^q \diff x =  \int_{A^{\rho_0}_{R_1}}\big((u-k_0)^+\big)^q \diff x \leq  \int_{A_{R_1}^{R_1+2\epsilon}}(u^+)^q \diff x.
\end{align*}
On the other hand, the inequality
\begin{align*}
\int_{A_{R_1}^{R_1+2\epsilon}} (u^+)^q \diff x \leq \bigg( 2C(\epsilon,\alpha,\beta,\gamma,\delta) k_{\ast} ^{-\frac{q(\bar{q}-q)}{\bar{q}}}\bigg)^{-\frac{1}{\delta_1}} \eta^{-\frac{1}{\delta_1^2}}
\end{align*}
is equivalent to 
\begin{equation*}
k_{\ast} \geq \big(2C(\epsilon,\alpha,\beta,\gamma,\delta) \big)^{\frac{\bar{q}}{q(\bar{q}-q)}} \eta^{\frac{\bar{q}}{\delta_1q(\bar{q}-q)}} \bigg(\int_{A_{R_1}^{R_1+2\epsilon}}(u^+)^q \diff x \bigg)^{\frac{\bar{q}\delta_1}{q(\bar{q}-q)}}.
\end{equation*}
We also have that the following inequality
\begin{equation*}
\int_{A_{R_1}^{R_1+2\epsilon}} (u^+)^q \diff x \leq \bigg(2C(\epsilon,\alpha,\beta,\gamma,\delta)k_{\ast}^{-\frac{q(\bar{q}-q)}{\bar{q}}}\bigg)^{-\frac{1}{\delta_2}} \eta^{-\frac{1}{\delta_1\delta_2} -\frac{\delta_2-\delta_1}{\delta_2^2}}
\end{equation*}
is equivalent to 
\begin{equation*}
k_{\ast} \geq \big(2C(\epsilon,\alpha,\beta,\gamma,\delta)\big)^{\frac{\bar{q}}{q(\bar{q}-q)}} \eta^{\big(\frac{1}{\delta_1}+\frac{\delta_2-\delta_1}{\delta_2}\big) \frac{\bar{q}}{q(\bar{q}-q)}} \bigg(\int_{A_{R_1}^{R_1+2\epsilon}} (u^+)^q \diff x\bigg)^{\frac{\bar{q}\delta_2}{q(\bar{q}-q)}}.
\end{equation*}
So if we choose
\begin{equation}\label{aprioriestkstar}
k_{\ast}= \bigg[1+(2C(\epsilon,\alpha,\beta,\gamma,\delta))^{\frac{\bar{q}}{q(\bar{q}-q)}} \eta^{\big(\frac{1}{\delta_1} +\frac{\delta_2-\delta_1}{\delta_2}\big)\frac{\bar{q}}{q(\bar{q}-q)}} \bigg]\bigg[1+\bigg(\int_{A_{R_1}^{R_1+2\epsilon}} |u|^q \diff x \bigg)^{\frac{\bar{q}\delta_2}{q(\bar{q}-q)}}\bigg],
\end{equation}
then, we obtain \eqref{{aprioriest16}}, and hence, thanks to Lemma \ref{leRecur}
\begin{equation*}
J_n= \int_{A_{R_1}^{R_1+2\epsilon}} \big((u-k_{n})^+\big)^q \chi_{A_{R_1}^{\rho_{n}}} \diff x \to 0 \text{ as } n \to \infty.
\end{equation*}
Note that, due to Lebesgue's dominated convergence theorem we have 
\[J_n \to \int_{A_{R_1}^{R_1+2\epsilon}}\big((u-k_{\ast})^+\big)^q \chi_{A_{R_1}^{R_1+\epsilon}} \diff x = \int_{A_{R_1}^{R_1+\epsilon}} \big((u-k_{\ast})^+\big)^q \diff x  \text{ as } n \to \infty. \]
Thus, $\int_{A_{R_1}^{R_1+\epsilon}} \big((u-k_{\ast})^+\big)^q \diff x=0$ and hence, $(u-k_{\ast})^+ = 0$ a.e. in $A_{R_1}^{R_1+\epsilon}$, i.e., 
\begin{equation}\esssup_{A_{R_1}^{R_1+\epsilon}} u \leq k_{\ast}.   \label{aprioriestkstar1}
\end{equation} Replacing $u$ by $-u$ in Steps 1 and 2 and arguing as above, we get 
\begin{equation} \label{aprioriestkstar2}
\esssup_{A_{R_1}^{R_1+\epsilon}} (-u) \leq k_{\ast}.
\end{equation}
It follows from \eqref{aprioriestkstar1} and \eqref{aprioriestkstar2} that 
\begin{equation}
\|u\|_{L^{\infty}(A_{R_1}^{R_1+\epsilon})} \leq k_{\ast}.   \label{aprioriestkstar3}
\end{equation}
Note that by Lemma \ref{well-defined}, we have 
\begin{equation}
u \in L^q\big(A_{R_1}^{R_1+2\epsilon}\big).  \label{aprioriestkstar4}
\end{equation}
Combining \eqref{aprioriestkstar}, \eqref{aprioriestkstar3} and \eqref{aprioriestkstar4} there exist $C>0$ and $\mu >0$, both independent of $u$, such that \eqref{mainestI} holds. This completes the proof of part (i).

(ii) We proceed in the same fashion as in part (i) of this proof. Let $u$ be a weak solution of problem \eqref{4defweaksol}. Without loss of generality we may assume that $t>\frac{q}{p}.$ Denote  $$ A:= \|L\|_{L^{\frac{q}{q-p}}\big(B(x_0,r_0)\big)}, \ B:= \|a\|_{L^{\frac{q}{q-p}}\big(B(x_0,r_0)\big)}  \text{ and } M:= \|b\|_{L^{\frac{t}{t-1}}\big(B(x_0,r_0)\big)},$$
and for $k>0,$ $\delta \in (0,r_0),$ denote
$$A_{k,\delta} := \{x \in B(x_0,\delta): u(x)>k\}.$$
We will prove that there exists a positive constant $C$ independent of $u$ such that for any $0<r_1<r_2<r_0$ and $k>0,$ the following Caccioppoli-type inequality holds true:
\begin{align}\label{Caccioppoli.Inq.2}
\int_{A_{k,r_1}}L(x) |\nabla u|^p \diff x &\leq C(A +Br_0^p) \bigg(\int_{A_{k,r_2}} \bigg(\frac{u-k}{r_2-r_1} \bigg)^q\diff x \bigg)^{\frac{p}{q}} + \notag \\  &\quad+pM\bigg( \int_{A_{k,r_2}} (u-k)^q\diff x \bigg)^{\frac{1}{q}}   |A_{k,r_2}|^{\frac{q-t}{qt}} +CB k^p |A_{k,r_2}|^{\frac{p}{q}}.    
\end{align}
Indeed, let $\xi \in C^{\infty}(\mathbb{R}^N),$ such that $\chi_{B(x_0,r_1)} \leq \xi \leq \chi_{B(x_0,r_2)}$ and $|\nabla \xi | \leq \frac{2}{r_2-r_1}.$ We note that for $\widetilde{u}\in \mathcal{D}_0^{1,p}(A_{R_1}^{R_2};L)$ and $\widetilde{\xi}\in C^1(\mathbb{R}^N)$ with $\chi_{B(x_0,r_1)} \leq \widetilde{\xi} \leq \chi_{{B(x_0,r_2)}},$ we have $\widetilde{u}\widetilde{\xi} \in \mathcal{D}_0^{1,p}(A_{R_1}^{R_2};L)$ and $\widetilde{u}\widetilde{\xi}$ is a test function for \eqref{P}. By this and Proposition~\ref{prop.(u-k)^+}, we can use $(u-k)^+ \xi^p$ as a test function in \eqref{P} and then repeating the arguments used in the proof of part (i), we easily obtain \eqref{Caccioppoli.Inq.2}.

Next, we define the recursive sequence $\{J_n\}$ as follows. For each $n \in \mathbb{N}_0$, define
 \[ J_n := \displaystyle\int_{A_{k_n,\rho_n}}(u-k_n)^q \diff x,\]
where
\[\rho_n:=\frac{r_0}{2} + \frac{r_0}{2^{n+1}}\  \ \text{and}\  \ k_n := k_\ast\left(1-\frac{1}{2^{n+1}}\right)\]
with $k_\ast>0$ to be specified later.
Note that 
\[ \rho_n  \downarrow \frac{r_0}{2},\ k_n \uparrow k_\ast \text{ and } \frac{r_0}{2} < \rho_n \leq r_0, \ \frac{k_\ast}{2} \leq k_n <k_\ast,\  \forall n\in \mathbb{N}_0.  \]
Denote $\bar{\rho}_n:= \frac{\rho_n +\rho_{n+1}}{2}$ ($n\in\mathbb{N}_0$) and fix $\zeta \in C^1(\mathbb{R})$, such that $\chi_{(-\infty, \frac{1}{2})} \leq \zeta \leq \chi_{(-\infty, \frac{3}{4})}$ and $|\zeta'| \leq 8$. Define 
\[\zeta_n(x):=\zeta\left(\frac{2^{n+1}}{r_0}\biggl(|x-x_0|-\frac{r_0}{2}\biggr)\right), \ \ x \in \mathbb{R}^N. \]
Then $\zeta_n\in C^1(\mathbb{R}^N)$, $\chi_{B(x_0, \rho_{n+1})}\leq\zeta_n \leq \chi_{B(x_0,\bar{\rho}_n)}$ and $ |\nabla \zeta_n| \leq \frac{2^{n+4}}{r_0}$ for all $n \in \mathbb{N}_0$.\\
Fix $\bar{q} \in (tp,p_s^{\ast})$. Using H\"oder's inequality, we have
\begin{align}
J_{n+1} = \int_{A_{k_{n+1},\rho_{n+1}}}(u-k_{n+1})^q \diff x \leq \left(\int_{A_{k_{n+1},\rho_{n+1}}}(u-k_{n+1})^{\bar{q}} \diff x\right)^{\frac{q}{\bar{q}}} \left|A_{k_{n+1},\rho_{n+1}}\right|^{\frac{\bar{q}-q}{\bar{q}}}.   \label{5thlocest}
\end{align}
It is easy to see that
\begin{align}
\int_{A_{k_{n+1},\rho_{n+1}}}(&u-k_{n+1})^{\bar{q}} \diff x \leq \int_{B(x_0,r_0)}\big((u-k_{n+1})^+\zeta_n\big)^{\bar{q}} \diff x.     \label{6thlocest}
\end{align}
By the assumption on $L$, $W^{1,p}(B(x_0,r_0); L):=\{u\in W^1(B(x_0,r_0)):\int_{B(x_0,r_0)}\big[|u|^p+L(x)|\nabla u|^p\big]\diff x<\infty\}$ is a Sobolev space with respect to the norm
$$\|u\|_{W^{1,p}(B(x_0,r_0); L)}:=\left(\int_{B(x_0,r_0)}\big[|u|^p+L(x)|\nabla u|^p\big]\diff x\right)^{\frac{1}{p}}.$$
Moreover, $W^{1,p}(B(x_0,r_0); L)\hookrightarrow  W^{1,p_s}(B(x_0,r_0))\hookrightarrow L^{\bar{q}}(B(x_0,r_0))$ in view of \cite[Theorem 1.3 and the embedding (1.22)]{Dra-Kuf-Nic}. Denote by $W^{1,p}_0(B(x_0,r_0); L)$ the closure of $C_c^\infty(B(x_0,r_0))$ in $W^{1,p}(B(x_0,r_0); L)$ with respect to the norm $\|\cdot\|_{W^{1,p}(B(x_0,r_0); L)}.$ For any $\widetilde{u} \in C_c^\infty(B(x_0,r_0))$ using the change of variable of the form $x=x_0+y,\ \ \widetilde{v}(y) = \widetilde{u}(x_0+y)$, and employing Sobolev's embedding and Poincar\'e's inequality we obtain
\begin{align*}
\bigg(\int_{B(x_0,r_0)}&|\widetilde{u}(x)|^{\bar{q}}\diff x\bigg)^{\displaystyle\frac{1}{\bar{q}}} = \left(\int_{B(0,r_0)}|\widetilde{v}(y)|^{\bar{q}}\diff y\right)^{\displaystyle\frac{1}{\bar{q}}}\\ 
&\leq C_1(r_0)\left( \int_{B(0,r_0)}\bigl(|\widetilde{v}(y)|^{p_s} +  |\nabla \widetilde{v}(y)|^{p_s} \bigr)\diff y \right)^{\displaystyle\frac{1}{p_s}}\\
&\leq C_2(r_0) \left(\int_{B(0,r_0)} |\nabla \widetilde{v}(y)|^{p_s} \diff y\right)^{\displaystyle\frac{1}{p_s}}=  C_2(r_0) \left(\int_{B(x_0,r_0)} |\nabla \widetilde{u}(x)|^{p_s} \diff x\right)^{\displaystyle\frac{1}{p_s}}\\ 
&\leq C_2(r_0) \left(\int_{B(x_0,r_0)} L^{-s}(x) \diff x \right)^{\displaystyle\frac{1}{sp}} \left(\int_{B(x_0,r_0)}L(x)|\nabla \widetilde{u}(x)|^p \diff x \right)^{{\displaystyle\frac{1}{p}}}.
\end{align*} 
Here, and in what follows, $C_i(r_0)$ ($i\in\mathbb{N}$) depend only on $r_0.$ Thus we obtain
$$\int_{B(x_0,r_0)}|\widetilde{u}(x)|^{\bar{q}}\diff x\leq C_3(r_0)D\left(\int_{B(x_0,r_0)}L(x)|\nabla \widetilde{u}(x)|^p \diff x, \right)^{{\displaystyle\frac{\bar{q}}{p}}}$$
where $D:= \biggl(\int_{B(x_0,r_0)}L^{-s}(x) \diff x\biggr)^{ \frac{\bar{q}}{sp}}$, for all $\widetilde{u} \in C_c^\infty(B(x_0,r_0)).$ By the density argument, it holds for all $\widetilde{u} \in W^{1,p}_0(B(x_0,r_0); L).$ It is easy to see that $(u-k_{n+1})^+\zeta_n\in W^{1,p}_0(B(x_0,r_0), L).$ Thus, applying the last inequality for $\widetilde{u}=(u-k_{n+1})^+\zeta_n$ and combining this with \eqref{6thlocest} we obtain
\begin{align*}
\int_{A_{k_{n+1},\rho_{n+1}}}(u-k_{n+1})^{\bar{q}} \diff x  &\leq C_3(r_0) D \Biggl(\int_{B(x_0,r_0)}L(x)|\nabla( (u-k_{n+1})^+\zeta_n)|^p \diff x \Biggr)^{{\frac{\bar{q}}{p}}}\\
&\leq C_4(r_0) D \biggl[\int_{B(x_0,r_0)}L(x)|\nabla(u-k_{n+1})^+|^p\zeta_n^p \diff x \\  &\hspace*{3cm} + \int_{B(x_0,r_0)}L(x)((u-k_{n+1})^+)^p|\nabla\zeta_n|^p \diff x \biggr]^{{\frac{\bar{q}}{p}}}\\
&\leq C_4(r_0) D \biggl[ \int_{A_{k_{n+1},\bar{\rho}_n}} L(x)|\nabla u|^p \diff x  \\&\hspace*{3cm}+ 2^{(n+4)p} r_0^{-p} \int_{A_{k_{n+1},\bar{\rho}_n}} L(x)(u-k_{n+1})^p \diff x\biggr]^{{\frac{\bar{q}}{p}}}\\
&\leq C_5(r_0)D \biggl[ \int_{A_{k_{n+1},\bar{\rho}_n}} L(x)|\nabla u|^p \diff x \\ &\hspace*{3cm}+ 2^{np} A \left(\int_{A_{k_{n+1},\bar{\rho}_n}}(u-k_{n+1})^q \diff x \right)^{\frac{p}{q}} \biggr]^{{\frac{\bar{q}}{p}}}.
\end{align*}
This yields 
\begin{equation}\label{7thlocest}
\int_{A_{k_{n+1},\rho_{n+1}}}(u-k_{n+1})^{\bar{q}} \diff x \leq C_5(r_0) D \Biggl[ \int_{A_{k_{n+1},\bar{\rho}_n}} L(x)|\nabla u|^p \diff x +2^{np} A J_n^{\frac{p}{q}} \Biggr]^{\frac{\bar{q}}{p}}.     
\end{equation}
 Applying \eqref{Caccioppoli.Inq.2} for $k=k_{n+1},\ r_1=\bar{\rho}_n$ and $r_2 = \rho_n$, we get
\begin{align*}
\int_{A_{k_{n+1},\bar{\rho}_{n}}} L(x)&|\nabla u|^p \diff x \leq C2^{(n+3)p} r_0^{-p}(A+Br_0^p) \bigg(\int_{A_{k_{n+1},\rho_n}} (u-k_{n+1})^q\bigg)^{\frac{p}{q}}  + \\ &+ pM \bigg( \int_{A_{k_{n+1},\rho_n}} (u-k_{n+1})^q\diff x \bigg)^{\frac{1}{q}} |A_{k_{n+1},\rho_n}|^{\frac{q-t}{qt}} +CB k_{\ast}^p |A_{k_{n+1},\rho_n}|^{\frac{p}{q}}.
\end{align*} 
Combining this and \eqref{7thlocest}, we obtain
\begin{align*}
\int_{A_{k_{n+1},\rho_{n+1}}} (u-k_{n+1})^{\bar{q}} \diff x \leq C_6(r_0) D \biggl[ (A+B)&2^{np}J_n^{\frac{p}{q}} +M J_n^{\frac{1}{q}} |A_{k_{n+1},\rho_n}|^{\frac{q-t}{qt}}\\
&+Bk_\ast^p \bigl|A_{k_{n+1},\rho_n}\bigr|^{\frac{p}{q}}+2^{np}AJ_n^{\frac{p}{q}}\biggr]^{\frac{\bar{q}}{p}}.
\end{align*}
From this and the estimate 
\begin{equation*}
\bigl|A_{k_{n+1},\rho_{n+1}}\bigr|\leq\bigl|A_{k_{n+1},\rho_n}\bigr| \leq \int_{A_{k_{n+1},\rho_n}} \biggl(\frac{u-k_n}{k_{n+1}-k_n}\biggr)^{q} \diff x  
\leq    \frac{2^{(n+2)q}}{k_\ast^q} J_n,
\end{equation*}
we obtain from \eqref{5thlocest} that
\begin{equation}\label{est-common2}
J_{n+1} \leq C_7(r_0) D^{\frac{q}{\bar{q}}} \biggl[(A+B) 2^{np}J_n^{\frac{p}{q}} + Mk_\ast^{1-\frac{q}{t}}2^{\frac{n(q-t)}{t}}J_n^{\frac{1}{q}+\frac{q-t}{qt}}\biggr]^{\frac{q}{p}} \frac{2^{\frac{nq(\bar{q}-q)}{\bar{q}}} J_n^{\frac{\bar{q}-q}{\bar{q}}} }{k_\ast^{\frac{q(\bar{q}-q)}{\bar{q}}}}.
\end{equation}
So if we choose $k_\ast>1$ then \eqref{est-common2} implies
\begin{align*}
J_{n+1} \leq C(A,B,M,D,r_0)  k_{\ast} ^{-\frac{q(\bar{q}-q)}{\bar{q}}} \eta ^n \bigg(J_n^{1+\delta_1}+J_n^{1+\delta_2}\bigg),
\end{align*}
where 
\begin{align*}
0<\delta_1:= \frac{q}{tp} -\frac{q}{\bar{q}}  < \delta_2 := \frac{\bar{q}-q}{\bar{q}}  \text{ and }
\eta:= 2^{q+\frac{q(\bar{q}-q)}{\bar{q}}} >1.
\end{align*}
Finally, arguing as in Step 3 of the proof of part (i) we get the desired conclusion.
\end{proof}

Obviously, Theorem \ref{Theorem.local.behaviorI} is a special case of Theorem~\ref{apriorimainthm}. Now we give the proof of Theorem \ref{Theorem.local.behaviorII}. Since this proof is similar to that of Theorem~\ref{apriorimainthm} (ii), we only sketch it.
\begin{proof}[Proof of Theorem \ref{Theorem.local.behaviorII}] Let $u$ be a solution of problem \eqref{1.1} and let $\mu\in (0,1-\frac{q}{p_s^\ast})$. We follow the argument in the proof of Theorem~\ref{apriorimainthm} (ii) with the choice $a=K$, $b=0$ and $\bar{q}:=\frac{q}{1-\mu}$ to obtain \eqref{est-common2} of the form
	\begin{equation*}
		J_{n+1} \leq C(r_0) D^{\frac{q}{\bar{q}}} \biggl[(A+B)2^{np}J_n^{\frac{p}{q}} \biggr]^{\frac{q}{p}} \frac{2^{\frac{nq(\bar{q}-q)}{\bar{q}}} J_n^{\frac{\bar{q}-q}{\bar{q}}} }{k_\ast^{\frac{q(\bar{q}-q)}{\bar{q}}}},
	\end{equation*}
where

$ A:= \|L\|_{L^{\frac{q}{q-p}}\big(B(x_0,r_0)\big)}, \ B:= \|K\|_{L^{\frac{q}{q-p}}\big(B(x_0,r_0)\big)}, D:= \biggl(\int_{B(x_0,r_0)}L^{-s}(x) \diff x\biggr)^{\frac{\bar{q}}{sp}},$
and $C(r_0)>0$  depends only on $r_0.$ This implies
\begin{equation}\label{Recur.Inq*}
J_{n+1} \leq C(r_0) D^{\frac{q}{\bar{q}}}(A+B)^{\frac{q}{p}} k_\ast^{-q\mu}\eta^n J_n^{1+\mu},
\end{equation}
where 
 $\eta:=2^{q(1+\mu)}>1.$ 
Invoking Lemma~\ref{leRecur} with $\delta_1=\delta_2=\mu$, we deduce from \eqref{Recur.Inq*} that  $J_n \to 0$ as $n \to \infty$, provided 
\begin{equation}
J_0 \leq \left[C(r_0) D^{\frac{q}{\bar{q}}}(A+B)^{\frac{q}{p}} k_\ast^{-q\mu}\right]^{-\frac{1}{\mu}} \eta^{-\frac{1}{\mu^2}}.   \label{{aprioriest16*}}
\end{equation}
We have 
\begin{align*}
J_0= \int_{A_{k_0,\rho_0}}(u-k_0)^q \diff x =  \int_{B(x_0,\rho_0)}\big((u-k_0)^+\big)^q \diff x \leq  \int_{B(x_0,r_0)}(u^+)^q \diff x.
\end{align*}
So if we choose
\begin{equation}\label{aprioriestkstar*}
k_{\ast}= \big[C(r_0)\eta^{\frac{1}{\mu}}\big]^{\frac{1}{q\mu}} D^{\frac{1}{\bar{q}\mu}} (A+B)^{\frac{1}{p\mu}}\bigg(\int_{B(x_0,r_0)} (u^+)^q\diff x \bigg)^{\frac{1}{q}},
\end{equation}
then, we obtain \eqref{{aprioriest16*}}, and hence, thanks to Lemma \ref{leRecur}
\begin{equation*}
J_n= \int_{B(x_0,r_0)} \big((u-k_{n})^+\big)^q \chi_{B(x_0,\rho_{n})} \diff x \to 0 \text{ as } n \to \infty.
\end{equation*}
Note that, due to Lebesgue's dominated convergence theorem we have 
\[J_n \to \int_{B(x_0,r_0)}\big((u-k_{\ast})^+\big)^q \chi_{B(x_0,\frac{r_0}{2})} \diff x = \int_{B(x_0,\frac{r_0}{2})} \big((u-k_{\ast})^+\big)^q \diff x  \text{ as } n \to \infty. \]
Thus, $\int_{B(x_0,\frac{r_0}{2})} \big((u-k_{\ast})^+\big)^q \diff x=0$ and hence, $(u-k_{\ast})^+ = 0$ a.e. in $B(x_0,\frac{r_0}{2})$, i.e., 
\begin{equation}\esssup_{B(x_0,\frac{r_0}{2})} u \leq k_{\ast}.   \label{aprioriestkstar1*}
\end{equation} Replacing $u$ by $-u$ in the arguments above, we get 
\begin{equation} \label{aprioriestkstar2*}
\esssup_{B(x_0,\frac{r_0}{2})} (-u) \leq k_{\ast}.
\end{equation}
It follows from \eqref{aprioriestkstar1*} and \eqref{aprioriestkstar2*} that 
\begin{equation}
\|u\|_{L^{\infty}(B(x_0,\frac{r_0}{2}))} \leq k_{\ast}.   \label{aprioriestkstar3*}
\end{equation}
Note that by Lemma \ref{well-defined}, we have 
\begin{equation}
u \in L^q\big(B(x_0,r_0)\big).  \label{aprioriestkstar4*}
\end{equation}
Combining \eqref{aprioriestkstar*}, \eqref{aprioriestkstar3*} and \eqref{aprioriestkstar4*} there exists $C=C(\mu,r_0)>0$ independent of $u$, such that \eqref{localmainestI} holds. The proof is complete.
\end{proof}


\subsection{The smoothness of solutions}

In this subsection we prove the results on smoothness of solutions.
\begin{proof}[Proof of Theorem~\ref{Theorem.C1-regularity1}]	
We rewrite \eqref{1.1} as
$$-L\operatorname{div}(|\nabla u|^{p-2}\nabla u)-	|\nabla u|^{p-2}(\nabla u\cdot \nabla L)=\lambda K|u|^{p-2}u,$$
i.e.,
$$-\operatorname{div}(|\nabla u|^{p-2}\nabla u)=\lambda\frac{K}{L}|u|^{p-2}u+	|\nabla u|^{p-2}(\nabla u\cdot \frac{\nabla L}{L}).$$
Thus, $\phi=u$ is a weak solution to
$$-\operatorname{div}\vec{a}(x,\phi,\nabla \phi)+b(x,\phi,\nabla \phi)=0,$$
where $\vec{a}(x,\phi,\nabla \phi)=-|\nabla \phi|^{p-2}\nabla \phi$ and $b(x,\phi,\nabla \phi)=\lambda\frac{K}{L}|u|^{p-2}u+	|\nabla \phi|^{p-2}(\nabla \phi\cdot \frac{\nabla L}{L}).$
In view of Corollary~\ref{2.corollary.local_embeddings} and Theorem~\ref{Theorem.local.behaviorII} we have $u\in W_{\loc}^{1,p}(A_{R_1}^{R_2})\cap L_{\loc}^\infty(A_{R_1}^{R_2}).$ Using Young's inequality, for any $R_1<r_1<r_2<R_2$ we have
$$|b(x,\phi,\nabla \phi)|\leq \lambda\|u\|_{L^\infty(A_{r_1}^{r_2})}^{p-1}\left|\frac{K}{L}\right|+\frac{p-1}{p}|\nabla \phi|^p+\frac{1}{p}\left|\frac{\nabla L}{L}\right|^p.$$
Hence
$$|b(x,\phi,\nabla \phi)|\leq \frac{p-1}{p}|\nabla \phi|^p+\left(\lambda\|u\|_{L^\infty(A_{r_1}^{r_2})}^{p-1}+\frac{1}{p}\right)\left(\left|\frac{K}{L}\right|+\left|\frac{\nabla L}{L}\right|^p\right).$$
Thus by \cite[Theorem 2 and its Remark]{DiBenedetto} we obtain $C_{\loc}^{1,\alpha}(A_{r_1}^{r_2})$ for any $R_2<r_1<r_2<R_2$ and hence the proof is completed. 
\end{proof}
Finally we conclude this subsection by proving H\"older regularity of eigenfunctions up to inner boundary. 
\begin{proof}[Proof of Theorem~\ref{Theorem.C1-regularity2}]
By Theorem~\ref{Theorem.local.behaviorI}, we have $u \in L^{\infty}\big(A_{R_1}^{R_1+\epsilon}\big)$. From this and the estimate 
\[\int_{A_{R_1}^{R_1+\epsilon}} |\nabla u|^p \diff x \leq \frac{1}{\underset{x\in A_{R_1}^{R_1+\epsilon}}{\essinf}\ L(x)} \int_{A_{R_1}^{R_1+\epsilon}} L(x) |\nabla u|^p \diff x < \infty,\]
we obtain $u \in W^{1,p}(A_{R_1}^{R_1+\epsilon}) \cap L^{\infty}\big(A_{R_1}^{R_1+\epsilon}\big)$. As in the proof of Theorem~\ref{Theorem.C1-regularity1}, we have 
$$-\operatorname{div}(|\nabla u|^{p-2}\nabla u)=\lambda\frac{K}{L}|u|^{p-2}u+	|\nabla u |^{p-2}(\nabla u\cdot \frac{\nabla L}{L}).$$ 
Thus $\phi=u\in W^{1,p}(A_{R_1}^{R_1+\epsilon}) \cap L^{\infty}\big(A_{R_1}^{R_1+\epsilon}\big)$ is a weak solution to the following problem
\begin{align*}
\begin{cases}
-\operatorname{div}(|\nabla \phi|^{p-2}\nabla \phi)=\lambda\frac{K}{L}|\phi|^{p-2} \phi+	|\nabla \phi|^{p-2}(\nabla \phi\cdot \frac{\nabla L}{L})\ \text{ in }\ A_{R_1}^{R_1+\epsilon},\\
\phi=0\  \text{ on }\ \partial B_{R_1}\ \text{ and }\ \phi = u \text{ on } \partial B_{R_1+\epsilon}.
\end{cases}
\end{align*}
By Theorem~\ref{Theorem.C1-regularity1}, we have $u \in C^{1,\alpha}(\partial B_{R_1+\epsilon})$. From this and $|\frac{\nabla L}{L}|+|\frac{K}{L}| \in L^{\infty}\big(A_{R_1}^{R_1+\epsilon}\big)$, we have $\phi =u \in C^{1,\beta_{\epsilon}}(\overline{A_{R_1}^{R_1+\epsilon}})$ for some $\beta_{\epsilon} \in (0,1)$ in view of \cite[Theorem 1]{Lieberman}.
\end{proof}
\subsection{Positivity and decay of solutions}
In this subsection we prove the positivity and decay of solutions. First, we prove Theorem~\ref{Theorem.everywhere-positivity}, which states that a nonnegative $C^1$ solution is positive everywhere.
\begin{proof}[Proof of Theorem~\ref{Theorem.everywhere-positivity}] By Theorem~\ref{Theorem.C1-regularity1}, we have $u\in C^1(A_{R_1}^{R_2}).$ The conclusion of the theorem then follows from \cite[Theorem 8.1]{Pucci-Serrin}. 
\end{proof}
Finally, we show the decay of solutions at infinity when the domain is unbounded.

	\begin{proof}[Proof of Corollary~\ref{Coro.decay.at.infinity.2}]
Denote $$\alpha_1:=\essinf_{x \in B_R^c} L(x) \text{ and } \beta_1:= \esssup_{x \in B_{R+r_0}^c} \bigg[\|L\|_{L^{\frac{q}{q-p}}(B(x,r_0))}+\|K\|_{L^{\frac{q}{q-p}}(B(x,r_0))}\bigg] ,$$
	then $0<\alpha_1,\beta_1<\infty$ by the assumptions of the corollary. Let $u$ be a solution to problem \eqref{1.1}.	We first  show that $u \in L^{p^\ast}(B_{R+\epsilon}^c)$ for all $\epsilon >0$. Indeed, fix an $\epsilon>0$ and let $\{u_n\} \subset C_c^1(B_{R_1}^c)$ such that 
		\begin{equation}\label{cordecay10}
		\int_{B_{R_1}^c} L(x) |\nabla u_n -\nabla u|^p \diff x \to 0 \text{ as } n \to \infty.
		\end{equation}  
		Since $\mathcal{D}^{1,p}_0(B_{R_1}^c;L) \hookrightarrow  L^p_{\loc}(B_{R_1}^c)$, up to a subsequence, we have 
		\begin{align}\label{cordecay11}
		\begin{cases}
		u_n \to u \text{ a.e. in } B_{R_1}^c,\\
		u_n \to u \text{ in } L^p(A_R^{R+\epsilon}).
		\end{cases}
		\end{align}
		Let $\phi \in C^{\infty}(\mathbb{R}^N)$ such that $\chi_{B^c_{R+\epsilon}} \leq \phi \leq \chi_{B^c_R}$ and $|\nabla \phi | \leq \frac{2}{\epsilon}$. Since $\phi u_n \in C_c^1(\mathbb{R}^N)$, by Sobolev's embedding we have 
		\[\bigg(\int_{\mathbb{R}^N} |\phi u_n|^{p^{\ast}} \diff x \bigg)^{\frac{p}{p^{\ast}}} \leq C \int_{\mathbb{R}^N} \big|\nabla(\phi u_n)\big|^p \diff x,\quad \forall n\in\mathbb{N},\]
		where $C>0$ is independent of $n.$ Hence 
		\begin{align*}
		\bigg(\int_{B^c_{R+\epsilon}} |u_n|^{p^{\ast}} \diff x \bigg)^{\frac{p}{p^{\ast}}} &\leq 2^{p-1}C \bigg(\int_{\mathbb{R}^N} \phi^p |\nabla u_n|^p \diff x+ \int_{\mathbb{R}^N} |u_n|^p|\nabla \phi|^p \diff x \bigg)\\
		&\leq 2^{p-1}C \bigg[\frac{1}{\alpha_1}\int_{B_R^c}L(x)|\nabla u_n|^p \diff x + \bigg(\frac{2}{\epsilon}\bigg)^p \int_{A_{R}^{R+\epsilon}} |u_n|^p \diff x \bigg].
		\end{align*}
		Letting $n\to\infty$ in the last estimate, using \eqref{cordecay10}, \eqref{cordecay11} and Fatou's lemma, we get
		\[\bigg(\int_{B_{R+\epsilon}^c}|u|^{p^{\ast}} \diff x\bigg)^{\frac{p}{p^{\ast}}} \leq \frac{2^{p-1}C}{\alpha_1} \int_{B_R^c} L(x)|\nabla u|^p \diff x + \frac{2^{2p-1}C}{\epsilon^p} \int_{A_R^{R+\epsilon}} |u|^p \diff x < \infty.\]
		Thus $u \in L^{p^{\ast}}(B_{R+\epsilon}^c).$ Hence for a fixed $x \in B_{R+r_0+\epsilon}$, we get 
		\begin{equation}\label{est.q-p*}
	\int_{B(x,r_0)} |u|^q \diff y \leq |B(x,r_0)|^{\frac{p^{\ast}-q}{p^{\ast}}} \bigg(\int_{B(x,r_0)} |u|^{p^{\ast}} \diff y\bigg)^{\frac{q}{p^{\ast}}}.
		\end{equation}
		Let $s>\frac{N}{p}+\frac{1}{p-1}$ be sufficiently large such that $q<p_s^\ast<p^\ast.$ Fix such $s$ and $\mu\in (0,1-\frac{q}{p_s^\ast})$. Clearly, all the assumptions of Theorem~\ref{Theorem.local.behaviorII} are satisfied so we obtain \eqref{localmainestI} for any ball $B(x,r_0)$. From this estimate and \eqref{est.q-p*}, for all $|x| > R+r_0+\epsilon$, we have
		\[ \|u\|_{L^{\infty}\big( B\big(x,\frac{r_0}{2}\big)\big)} \leq C(r_0,\mu)(\alpha_1^{-s}|B(x,r_0)|)^{\frac{1}{sp\mu}} \beta_1^{\frac{1}{\mu p}} |B(0,r_0)|^{\frac{p^{\ast}-q}{qp^{\ast}}} \bigg(\int_{B(x,r_0)} |u|^{p^{\ast}} \diff y\bigg)^{\frac{1}{p^{\ast}}}. \]
		That is, 
		\[ \|u\|_{L^{\infty}\big( B\big(x,\frac{r_0}{2}\big)\big)} \leq C(r_0,\mu,\alpha_1,\beta_1) \bigg(\int_{B(x,r_0)} |u|^{p^{\ast}} \diff y\bigg)^{\frac{1}{p^{\ast}}}. \]
		where $C(r_0,\mu,\alpha_1,\beta_1)$ is independent of $x.$ Since $u \in L^{p^{\ast}}(B_{R+\epsilon}^c)$, we deduce from the last inequality that $u(x) \to 0$ uniformly as $|x| \to \infty$.
	\end{proof}

\section{The asymptotic estimates of solutions towards the boundary}\label{Sec.Behavior}
	
In this section we prove the asymptotic estimates of solutions towards the boundary stated in Theorems~\ref{radialsoldecayleft} and \ref{radialsoldecayright}. Such asymptotic estimates are obtained due to strengthened versions of $\mathrm{(A)}$ near $R_1$ and $R_2.$ 
\begin{remark}\label{Rmk.non-oscillation}\rm
Note that in the condition $\mathrm{(A)}$, when $v,w\in L^1_{\loc}(R_1,R_2)$ and $P(r)<\infty$ for all $r\in(R_1,R_2),$ then $\int_{R_1}^{R_2}P(r)\sigma(r)\diff r<\infty$ is equivalent to $\int_{R_1}^{r_1}P(r)\sigma(r)\diff r<\infty$ and $\int_{r_2}^{R_2}P(r)\sigma(r)\diff r<\infty$ for some $R_1<r_1<r_2<R_2.$	Note that $\mathrm{(A_{\epsilon,L})}$ implies that $\int_{R_1}^{r_1}P(r)\sigma(r)\diff r<\infty$ for some $r_1\in (R_1,\xi).$ Indeed, since $\rho^{1-p'}\in L^1(R_1;\xi)$, we have $\int_{R_1}^{r}\rho^{1-p'}(\tau)\diff \tau\to 0$ as $r\to R_1^+.$ Thus, there exists $r_1\in (R_1,\xi)$ such that $P(r)=\left(\int_{R_1}^{r}\rho^{1-p'}(\tau)\diff \tau\right)^{p-1}$ for all $r\in (R_1,r_1).$ Hence, by $\mathrm{(A_{\epsilon,L})}$ we have
	$$P(r)<C^{\frac{p-1}{\epsilon}}\left(\int_{r}^{\xi}\sigma(\tau)\diff \tau \right)^{-\frac{p-1}{\epsilon}},\quad \forall r\in (R_1,r_1).$$
	Therefore
	\begin{align*}
	\int_{R_1}^{r_1}P(r)\sigma(r)\diff r&<C^{\frac{p-1}{\epsilon}}\int_{R_1}^{r_1}\left(\int_{r}^{\xi}\sigma(\tau)\diff \tau \right)^{-\frac{p-1}{\epsilon}}\sigma(r)\diff r\\
	&<\frac{C^{\frac{p-1}{\epsilon}}\epsilon}{p-1-\epsilon}\left(\int_{r_1}^{\xi}\sigma(\tau)\diff \tau \right)^{-\frac{p-1-\epsilon}{\epsilon}}<\infty.
	\end{align*}
Similarly, 	it is easy to see that $\mathrm{(A_{\epsilon,R})}$ implies that $\int_{r_2}^{R_2}P(r)\sigma(r)\diff r<\infty$ for some $r_2\in (\xi,R_2).$ 
\end{remark}


 We start the proof of Theorem~\ref{radialsoldecayleft} by stating nonoscillatory property of the radial solution in the right neighborhood of $R_1.$ This fact can be obtained by applying \cite[Theorem 1.14]{Opic-Kufner} and using a similar argument to that of \cite[Proof of Proposition 4.3]{Drabek-Kuliev}. Therefore, we omit it.
\begin{lemma}[Nonoscillatory I]\label{non-oscillation.L}
	Assume that $\mathrm{(A_{\epsilon,L})}$ holds. Then for a solution $u\in C^1(R_1,R_2)$ of \eqref{ode} with $u(R_1)=0,$ there exists $a\in (R_1,\xi)$ such that $u(r)\ne 0$ and $u'(r)\ne 0$ for all $r\in (R_1,a).$
\end{lemma}
Thanks to Lemma~\ref{non-oscillation.L} and the technique used in \cite[Proof of Theorem 1.1]{Dra-Kuf-Kul}, we now prove the behavior of $u(x)$ and $\nabla u(x)$ as $|x|\to R_1^+$, 
provided hypothesis of Theorem \ref{radialsoldecayleft} is satisfied.
	\begin{proof}[Proof of Theorem~\ref{radialsoldecayleft}]
	Since $u\in C^1(R_1,R_2)$ and $u(R_1)=u(R_2)=0$, there exists $r_0\in (R_1,R_2)$ such that $u'(r_0)=0.$ Take $\tilde{a}:=\min\{r \in (R_1,R_2): u'(r)=0\}$. Then, $\tilde{a}\in (R_1,R_2)$ in view of Lemma~\ref{non-oscillation.L}. Clearly, $u(r)$ satisfies 
	\begin{align*}
	\begin{cases}
	-(\rho(r)|u'(r)|^{p-2}u'(r))'= \lambda \sigma(r)|u(r)|^{p-2}u(r), \ \ r \in (R_1,R_2),\\
	u(R_1)=0=u'(\tilde{a}).
	\end{cases}
	\end{align*}
	Then, \[\rho(r)|u'(r)|^{p-2}u'(r) = \lambda \int_r^{\tilde{a}} \sigma(r)|u(\tau)|^{p-2}u(\tau)\diff \tau, \ \ \forall r \in (R_1,\tilde{a}).. \]
	We may assume that $u'(r)>0$ in $(R_1,\tilde{a})$ and hence $u(r)>0$ in $(R_1,\tilde{a})$.	
	Thus, we have 
	\begin{equation}  \label{derivativesolr1}
	u'(r)= \lambda^{\frac{1}{p-1}}\rho^{1-p'}(r)\left( \int_r^{\tilde{a}} \sigma(\tau) u^{p-1}(\tau) \diff \tau\right)^{\frac{1}{p-1}}, \ \ \forall r \in (R_1,\tilde{a}).
	\end{equation}
	Hence 
	\begin{align}
	u(r)= \lambda^{\frac{1}{p-1}} \int_{R_1}^r \rho^{1-p'}(t)\left( \int_t^{\tilde{a}} \sigma(\tau) u^{p-1}(\tau) \diff \tau\right)^{\frac{1}{p-1}} \diff t, \ \ \forall r \in (R_1,\tilde{a}). \label{estextsol3}
	\end{align}
	\textbf{Estimates from below:}	Fix $a \in (R_1,\tilde{a})$, then 
	\begin{equation*}
	u(r) \geq \lambda^{\frac{1}{p-1}} \biggl(\int_{a}^{\tilde{a}} \sigma(\tau)u^{p-1}(\tau) \diff \tau \biggr)^{\frac{1}{p-1}} \int_{R_1}^r \rho^{1-p'}(t) \diff t , \ \ \forall r \in (R_1,a),
	\end{equation*}
	i.e.,
	$$	u(r) \geq C_1 \int_{R_1}^r \rho^{1-p'}(t) \diff t, \ \ \forall r \in (R_1,a),$$
where $C_1:=\lambda^{\frac{1}{p-1}} \biggl(\int_{a}^{\tilde{a}} \sigma(\tau)u^{p-1}(\tau) \diff \tau \biggr)^{\frac{1}{p-1}}.$
	 
	To obtain an estimate from below of the derivative of solution, we use  \eqref{derivativesolr1} to get
	\begin{equation*}
	u'(r) \geq \lambda^{\frac{1}{p-1}} \rho^{1-p'}(r) \bigg(\int_{a}^{\tilde{a}} \sigma(\tau)u^{p-1}(\tau)\diff \tau\bigg)^{\frac{1}{p-1}}, \ \forall r \in (R_1,a),
	\end{equation*}
	i.e.,
	\[u'(r) \geq C_1 \rho^{1-p'}(r), \ \ \forall r \in (R_1,a).\] 
	
	\noindent\textbf{Estimates from above:} We proceed with an iteration argument.
	
	\textbf{1st Step:} From \eqref{estextsol3} and H\"older's inequality, for all $r \in (R_1, \tilde{a})$, we have 
	\begin{align*}
	u(r) &\leq \lambda^{\frac{1}{p-1}} \int_{R_1}^r \rho^{1-p'} (t)\left( \int_t^{\tilde{a}} \sigma(\tau) \diff \tau \right)^{\frac{1}{p(p-1)}}  \left( \int_t^{\tilde{a}} \sigma(\tau)u^p(\tau) \diff \tau\right)^{\frac{1}{p}} \diff t\\
	&\leq \lambda^{\frac{1}{p-1}} \left( \int_{R_1}^{\tilde{a}} \sigma(\tau)u^p(\tau) \diff \tau\right)^{\frac{1}{p}} \int_{R_1}^r \rho^{1-p'}(t) \left( \int_t^{\tilde{a}} \sigma(\tau) \diff \tau \right)^{\frac{1}{p(p-1)}} \diff t,
	\end{align*}
	i.e., 
	\begin{equation}
	u(r) \leq c_1 \int_{R_1}^r \rho^{1-p'}(t) I_1^{\frac{1}{p-1}}(t)\diff t, \ \ \forall r \in (R_1,\tilde{a}),    \label{estextsol5}
	\end{equation}
	where $c_1:= \lambda^{\frac{1}{p-1}}\left( \int_{R_1}^{\tilde{a}} \sigma(\tau) u^p(\tau) \diff \tau \right) ^{\frac{1}{p}}$ and 
	\begin{equation*}
	I_1(t):= \left( \int_t^{\tilde{a}}\sigma(\tau) \diff \tau\right)^{\frac{1}{p}}, \ \ \forall t \in (R_1,\tilde{a}).
	\end{equation*}
	Here we note that $c_1\in (0,\infty)$ since $$\int_{A_{R_1}^{R_2}}w(|x|)|u|^p\diff x=\frac{1}{\lambda}\int_{A_{R_1}^{R_2}}v(|x|)|\nabla u|^p\diff x<\infty.$$
	
	\textbf{2nd Step:} Using \eqref{estextsol5} in \eqref{estextsol3}, we get
	\begin{align*}
	u(r) \leq \lambda^{\frac{1}{p-1}} \int_{R_1}^r \rho^{1-p'} (t) \left[ \int_t^{\tilde{a}}\sigma(\tau) \left( c_1 \int_{R_1}^{\tau} \rho^{1-p'}(t_1) I_1^{\frac{1}{p-1}}(t_1) dt_1\right)^{p-1} \diff \tau\right]^{\frac{1}{p-1}}\diff t
	\end{align*}
		i.e.,
		$$u(r) \leq c_2 \int_{R_1}^r \rho^{1-p'}(t) I_2^{\frac{1}{p-1}}(t) \diff t, \ \ \forall r \in (R_1,\tilde{a}),$$
	where $c_2:= \lambda^{\frac{1}{p-1}} c_1$ and 
	\begin{equation*}
I_2(t):= \int_t^{\tilde{a}}\sigma(\tau) \left(\int_{R_1}^{\tau} \rho^{1-p'}(t_1) I_1^{\frac{1}{p-1}}(t_1) \diff t_1\right)^{p-1} \diff \tau.
	\end{equation*}
	
	\textbf{nth Step:} By induction, we obtain the following estimate for arbitrary $n$,
	\begin{equation}
	u(r) \leq c_n \int_{R_1}^r \rho^{1-p'}(t) I_n^{\frac{1}{p-1}}(t) \diff t, \ \ \forall r \in (R_1,\tilde{a}) ,    \label{estextsol6}
	\end{equation} 
	where $c_n:= \lambda^{\frac{1}{p-1}}c_{n-1}$ and 
	\begin{equation}\label{In}
I_n(t):= \int_t^{\tilde{a}}\sigma(\tau) \left(\int_{R_1}^{\tau} \rho^{1-p'}(t_{n-1}) I_{n-1}^{\frac{1}{p-1}}(t_{n-1}) \diff t_{n-1}\right)^{p-1} \diff \tau, \ \ \forall t \in (R_1,\tilde{a}).
	\end{equation}
	By \eqref{estextsol6}, to prove upper estimate for the solution $u$ near $\partial B_{R_1}$ it is sufficient to show that there exists $n \in \mathbb{N}$ and a constant $C>0$ such that 
	\[I_n(t) < C, \ \ \forall t \in (R_1,\tilde{a}).\] 
	To this end, fix $\tilde{\xi}\in(\max\{\xi,\tilde{a}\},R_2),$  where $\xi$ appears in $\mathrm{(A_{\epsilon,L})}$. By $\mathrm{(W)}$ and $\mathrm{(A_{\epsilon,L})},$ there exists a constant $\bar{C}>0$ such that
	\begin{equation}
	\left( \int_r^{\tilde{\xi}} \sigma(\tau)\diff \tau \right) \left( \int_{R_1}^r \rho^{1-p'}(\tau) \diff \tau\right)^{\epsilon} < \bar{C}, \ \ \forall r \in (R_1, \tilde{\xi}).   \label{estextsol7}
	\end{equation}
	Indeed, for $r \in (R_1,\xi)$, we have 
	\begin{align*}
	&\left( \int_r^{\tilde{\xi}} \sigma(\tau)\diff \tau \right) \left( \int_{R_1}^r \rho^{1-p'}(\tau) \diff \tau\right)^{\epsilon}\\ 
	&\quad=	\left( \int_r^{\xi} \sigma(\tau)\diff \tau \right) \left( \int_{R_1}^r \rho^{1-p'}(\tau) \diff \tau\right)^{\epsilon} + 	\left( \int_{\xi}^{\tilde{\xi}} \sigma(\tau)\diff \tau \right) \left( \int_{R_1}^r \rho^{1-p'}(\tau) \diff \tau\right)^{\epsilon} \\
	&\quad	\leq C + 	\left( \int_{\xi}^{\tilde{\xi}} \sigma(\tau)\diff \tau \right) \left( \int_{R_1}^{\xi}\rho^{1-p'}(\tau) \diff s\right)^{\epsilon} := \bar{C}_1. 
	\end{align*}
	For $r \in [\xi, \tilde{\xi}]$, we have 
	\begin{align*}
	\biggl(&\int_r^{\tilde{\xi}} \sigma(\tau) \diff \tau \biggr) \biggl(\int_{R_1}^r \rho^{1-p'}(\tau) \diff \tau\biggr)^{\epsilon} 
	&\leq   \biggl(\int_{\xi}^{\tilde{\xi}} \sigma(\tau) \diff \tau \biggr) \biggl(\int_{R_1}^{\tilde{\xi}}\rho^{1-p'}(\tau) \diff \tau\biggr)^{\epsilon} =: \bar{C}_2.
	\end{align*}
	Take $\bar{C} = \max \{\bar{C}_1, \bar{C}_2\}$, we obtain \eqref{estextsol7}.
	
	\noindent Similarly, we may also assume that \eqref{estextsol7} holds for  $\epsilon$ satisfying
	\begin{equation}\label{estextsol8}
		\epsilon \neq \frac{kp(p-1)}{kp+1}, \forall k \in \mathbb{N}_0\quad \text{i.e.,}\quad \frac{1}{p}-k\frac{p-1-\epsilon}{\epsilon}\ne 0, \ \forall k \in \mathbb{N}_0.   
		\end{equation}
		We now use \eqref{estextsol7} and \eqref{estextsol8}
		 to estimate $I_n(t)$. Let $n_0\in \mathbb{N}$ such that
		\begin{align*}
		\frac{\epsilon}{p(p-1-\epsilon)} +1<n_0 < \frac{\epsilon}{p(p-1-\epsilon)} +2,
		\end{align*}
		i.e., $n_0$ is the smallest integer $n$ such that 
		\[\frac{1}{p} - (n-1)\frac{p-1-\epsilon}{\epsilon}<0.\]
Clearly, $n_0\geq 2.$ We first prove the following estimate for $I_n.$

\textbf{Claim 1.} For each $n\in \{1,\cdots, n_0-1\},$ there exists $\tilde{c}_n>0$ such that
\begin{equation}\label{induction.est.In}
I_n(t) \leq \tilde{c}_{n} \left[\int_t^{\tilde{\xi}} \sigma (\tau) \diff \tau\right] ^{\frac{1}{p}-(n-1)\frac{p-1-\epsilon}{\epsilon}},\ \  \forall t \in (R_1,\tilde{a}).
\end{equation}
We prove the Claim 1 by induction. The conclusion is obvious if $n_0=2.$ Suppose that $n_0\geq 3$ and \eqref{induction.est.In} holds for some $n$ with $1\leq n<n_0-1.$ We prove that \eqref{induction.est.In} holds for $n+1$ too. Indeed, from \eqref{In}, \eqref{estextsol7}  and \eqref{induction.est.In} we have 
\begin{align}
I_{n+1}(t) &\leq \int_t^{\tilde{a}}\sigma(\tau) \biggl[\int_{R_1}^{\tau} \rho^{1-p'}(t_{n}) \tilde{c}^{\frac{1}{p-1}}_{n}\bigg(\int_{t_{n}}^{\tilde{\xi}}\sigma(t_{n-1}) \diff t_{n-1} \bigg)^{\frac{1}{p(p-1)} -  \frac{(n-1)(p-1-\epsilon)}{\epsilon(p-1)}}\diff t_{n}\biggr]^{p-1} \diff \tau \notag \\
& \leq \tilde{c}_{n+1}^1 \int_t^{\tilde{a}}\sigma(\tau) \biggl[\int_{R_1}^{\tau} \rho^{1-p'}(t_{n}) \bigg(\int_{R_1}^{t_{n}} \rho^{1-p'}(t_{n-1}) \diff t_{n-1} \bigg)^{-\frac{\epsilon}{p(p-1)}+\frac{(n-1)(p-1-\epsilon)}{p-1}} \diff t_{n}\biggr]^{p-1} \diff \tau\notag\\
&= \tilde{c}_{n+1}^2 \int_t^{\tilde{a}} \sigma(\tau) \left(\int_{R_1}^{\tau}\rho^{1-p'}(t_{n}) \diff t_{n}\right)^{-\frac{\epsilon}{p}+(n-1)(p-1-\epsilon)+p-1 }\diff \tau, \ \ \forall t \in (R_1,\tilde{a}).   \label{In+1}
\end{align}
Here we note that $ -\frac{\epsilon}{p(p-1)}+\frac{(n-1)(p-1-\epsilon)}{p-1}+1\geq 1-\frac{\epsilon}{p-1}>0.$
From  \eqref{estextsol7}, \eqref{In+1} and noticing $\frac{1}{p}-\frac{n(p-1-\epsilon)}{\epsilon}>0$, we have 
\begin{align*}
I_{n+1}(t) &\leq  \tilde{c}_{n+1}^3 \int_t^{\tilde{a}} \sigma(\tau) \left[ \int_{\tau}^{\tilde{\xi}} \sigma(t_{n})\diff t_{n}\right]^{\frac{1}{p}-\frac{(n-1)(p-1-\epsilon)}{\epsilon}-\frac{p-1}{\epsilon}} \diff \tau\\
&=-\tilde{c}_{n+1}^3 \int_{t}^{\tilde{a}} \left[\int_{\tau}^{\tilde{\xi}} \sigma(t_{n}) \diff t_{n}\right]^ {\frac{1}{p}-\frac{(n-1)(p-1-\epsilon)}{\epsilon}-\frac{p-1}{\epsilon}}  \diff \left(\int_{\tau}^{\tilde{\xi}} \sigma(t_{n}) \diff t_{n}\right)\\
&= -\frac{\tilde{c}_{n+1}^3}{\frac{1}{p}-\frac{n(p-1-\epsilon)}{\epsilon}} \left[\int_{\tau}^{\tilde{\xi}} \sigma(t_{n}) \diff t_{n}\right]^{\frac{1}{p}-\frac{n(p-1-\epsilon)}{\epsilon}}\bigg| _{\tau=t}^{\tau=\tilde{a}}\\
&\leq \tilde{c}_{n+1} \left[\int_{t}^{\tilde{\xi}} \sigma(t_{n}) \diff t_{n}\right]^{\frac{1}{p}-\frac{n(p-1-\epsilon)}{\epsilon}}, \ \ \forall t \in (R_1,\tilde{a}),
\end{align*}
where  $$\tilde{c}_{n+1} := \frac{\tilde{c}^3_{n+1}}{\frac{1}{p}-\frac{n(p-1-\epsilon)}{\epsilon}}.$$
Therefore, \eqref{induction.est.In} also holds for $n+1$ and hence, Claim 1 is proved.

\textbf{Claim 2.} There exists $\tilde{c}_{n_0}>0$ such that $I_{n_0}(t) < \tilde{c}_{n_0}$ for all $t\in (R_1,\tilde{a})$.

\noindent Indeed, from \eqref{In}, \eqref{estextsol7} and applying \eqref{induction.est.In} for $n=n_0-1$, we obtain 
\begin{align*}
I_{n_0}(t) &\leq \int_t^{\tilde{a}}\sigma(\tau) \biggl[\int_{R_1}^{\tau} \rho^{1-p'}(t_{n_0-1}) \tilde{c}^{\frac{1}{p-1}}_{n_0-1}\times \notag\\   &\ \ \ \  \qquad \times\bigg(\int_{t_{n_0-1}}^{\tilde{\xi}}\sigma(t_{n_0-2}) \diff t_{n_0-2} \bigg)^{\frac{1}{p(p-1)} -  \frac{(n_0-2)(p-1-\epsilon)}{\epsilon(p-1)}}\diff t_{n_0-1}\biggr]^{p-1} \diff \tau \notag \\
& \leq \tilde{c}_{n_0}^1 \int_t^{\tilde{a}}\sigma(\tau) \biggl[\int_{R_1}^{\tau} \rho^{1-p'}(t_{n_0-1})\times \notag\\ &\ \ \ \ \qquad \times \bigg(\int_{R_1}^{t_{n_0-1}} \rho^{1-p'}(t_{n_0-2}) \diff t_{n_0-2} \bigg)^{-\frac{\epsilon}{p(p-1)}+\frac{(n_0-2)(p-1-\epsilon)}{p-1}} \diff t_{n_0-1}\biggr]^{p-1} \diff \tau 
\end{align*}
for all $ t \in (R_1,\tilde{a}).$ 
Taking into account $-\frac{\epsilon}{p(p-1)}+\frac{(n_0-2)(p-1-\epsilon)}{p-1}+1\geq 1-\frac{\epsilon}{p-1}>0$, we obtain from the last estimate that there exists $\tilde{c}_{n_0}^2>0 $ such that
\begin{equation}\label{In0}
I_{n_0}(t)\leq \tilde{c}_{n_0}^2 \int_t^{\tilde{a}} \sigma(\tau) \left(\int_{R_1}^{\tau}\rho^{1-p'}(t_{n_0-1}) \diff t_{n_0-1}\right)^{-\frac{\epsilon}{p}+(n_0-2)(p-1-\epsilon)+p-1 }\diff \tau, \ \ \forall t \in (R_1,\tilde{a}).   
\end{equation}
From  \eqref{estextsol7}, \eqref{In0} and noticing $\frac{1}{p}-(n_0-1)\frac{p-1-\epsilon}{\epsilon}<0$, there is $\tilde{c}_{n_0}^3 $ such that
\begin{align*}
I_{n_0}(t) &\leq  \tilde{c}_{n_0}^3 \int_t^{\tilde{a}} \sigma(\tau) \left[ \int_{\tau}^{\tilde{\xi}} \sigma(t_{n_0-1})\diff t_{n_0-1}\right]^{\frac{1}{p}-\frac{(n_0-1)(p-1-\epsilon)}{\epsilon}} \diff \tau\\
&=-\tilde{c}_{n_0}^3 \int_{t}^{\tilde{a}} \left[\int_{\tau}^{\tilde{\xi}} \sigma(t_{n_0-1}) \diff t_{n_0-1}\right]^ {\frac{1}{p}-\frac{(n_0-2)(p-1-\epsilon)}{\epsilon}-\frac{p-1}{\epsilon}}  \diff \left(\int_{\tau}^{\tilde{\xi}} \sigma(t_{n_0-1}) \diff t_{n_0-1}\right)\\
&= -\frac{\tilde{c}_{n_0}^3}{\frac{1}{p}-\frac{(n_0-1)(p-1-\epsilon)}{\epsilon}} \left[\int_{\tau}^{\tilde{\xi}} \sigma(t_{n_0-1}) \diff t_{n_0-1}\right]^{\frac{1}{p}-\frac{(n_0-1)(p-1-\epsilon)}{\epsilon}}\bigg| _{\tau=t}^{\tau=\tilde{a}}\\
&\leq -\frac{\tilde{c}_{n_0}^3}{\frac{1}{p}-\frac{(n_0-1)(p-1-\epsilon)}{\epsilon}} \left[\int_{\tilde{a}}^{\tilde{\xi}} \sigma(t_{n_0-1}) \diff t_{n_0-1}\right]^{\frac{1}{p}-\frac{(n_0-1)(p-1-\epsilon)}{\epsilon}} =: \tilde{c}_{n_0}, \ \ \forall t \in (R_1,\tilde{a}).
\end{align*}
Thus, we have proved Claim 2.

\noindent By Claim 2, we get from \eqref{estextsol6} that
	\begin{equation}\label{est.u.above}
	u(r) \leq  C_2  \int_{R_1}^r \rho^{1-p'}(t) \diff t,\quad \forall r \in (R_1,\tilde{a}),
	\end{equation}
	where $C_2:=c_{n_0} \tilde{c}^{\frac{1}{p-1}}_{n_0}.$

	Finally, we look for the estimate of $u'$ from above. By \eqref{estextsol7} and \eqref{est.u.above}, we have 
	\begin{align*}
	\int_{r}^{\tilde{a}} \sigma(\tau) u^{p-1}(\tau) \diff \tau &\leq C_2^{p-1} \int_{t}^{\tilde{a}} \sigma(\tau) \bigg(\int_{R_1}^{\tau} \rho^{1-p'}(t) \diff t\bigg)^{p-1} \diff \tau\\
	&\leq \bar{C}_2 \int_{r}^{\tilde{a}} \sigma (\tau)\bigg(\int_{\tau}^{\tilde{\xi}} \sigma(t) \diff t \bigg)^{-\frac{p-1}{\epsilon}} \diff \tau\\
	&= \bar{C}_2 \bigg(\frac{\epsilon}{p-1-\epsilon}\bigg) \bigg(\int_{\tau}^{\tilde{\xi}} \sigma(t) \diff t \bigg)^{-\frac{p-1-\epsilon}{\epsilon}}\bigg|_{\tau=t}^{\tilde{a}}\\
	&\leq\frac{\epsilon\bar{C}_2 }{p-1-\epsilon} \bigg(\int_{\tilde{a}}^{\tilde{\xi}} \sigma(\tau) \diff \tau\bigg)^{-\frac{p-1-\epsilon}{\epsilon}},\quad \forall r\in (R_1,\tilde{a}).
	\end{align*}
Combining this and \eqref{derivativesolr1} we deduce 
	\[u'(r) \leq \tilde{C}_2 \rho^{1-p'}(r), \ \ \forall r \in (R_1,a).\]  
\end{proof}	
The asymptotic estimates of solutions towards the boundary $\partial B_{R_2}$ are obtained in the same manner. As before, we need the following nonoscillatory property and its proof can be obtained by invoking \cite[Theorem 6.2]{Opic-Kufner} and using a similar argument to that of \cite[Proof of Proposition 4.3]{Drabek-Kuliev}. Therefore, we omit it. 
\begin{lemma}[Nonoscillatory II]\label{non-oscillation.R}
	Assume that $\mathrm{(A_{\epsilon,R})}$ holds. 	Then for a solution $u\in C^1(R_1,R_2)$ of \eqref{ode} with $u(R_2)=0,$ there exists $b\in (\xi,R_2)$ such that $u(r)\ne 0$ and $u'(r)\ne 0$ for all $r\in (b,R_2).$
\end{lemma}
Using Lemma~\ref{non-oscillation.R} and similar argument as in the proof of Theorem~\ref{radialsoldecayleft} we prove Theorem~\ref{radialsoldecayright} as follows.
\begin{proof}[Proof of Theorem~\ref{radialsoldecayright}]

Let $\tilde{b}:=\max \{r \in (R_1,R_2): u'(r)=0\}$. Then $\tilde{b}\in (R_1,R_2)$ in view of Lemma~\ref{non-oscillation.R}. 	We have $u \in C^1(R_1,R_2)$ satisfies 
\begin{align*}
\begin{cases}
-(\rho(r)|u'(r)|^{p-2}u'(r))' = \lambda \sigma(r) |u(r)|^{p-2}u(r), \ \ r \in (R_1,R_2),\\
u'(\tilde{b})=0=u(R_2).
\end{cases}
\end{align*}
We may assume that $u'(r)<0$ in $(\tilde{b},R_2)$ and hence $u(r)>0$ in $(\tilde{b},R_2)$.	
Thus, we have 
\[-u'(r)= \lambda^{\frac{1}{p-1}} \rho^{1-p'} (r) \left(\int_{\tilde{b}}^{r}\sigma(t) u^{p-1}(t) \diff t \right)^{\frac{1}{p-1}},\quad r\in(\tilde{b},R_2).\]
Using this and the fact that $u(R_2)=0$, we get 
\begin{equation*}
u(r)= \lambda^{\frac{1}{p-1}} \int_{r}^{R_2} \rho^{1-p'}(t)\left( \int_{\tilde{b}}^t \sigma(\tau) u^{p-1}(\tau) \diff \tau\right)^{\frac{1}{p-1}} \diff t, \ \ \forall r \in (\tilde{b},R_2).
\end{equation*}
The rest of the proof is similar to that of the proof of Theorem~\ref{radialsoldecayleft} for which we mofdify $\mathrm{(A_{\epsilon,R})}$ as 
\begin{align*}
\left(\int_{\tilde{\xi}}^{r} \sigma(\tau) \diff \tau \right) \left( \int_{r}^{R_2} \rho^{1-p'}(\tau) \diff \tau\right)^{\epsilon} &\leq \bar{C},\quad \forall r\in (\tilde{\xi},R_2)
\end{align*}
for some fixed $\tilde{\xi} \in (R_1, \min\{\tilde{b},\xi\})$.
\end{proof}


\section{Applications}\label{Sec.Applications}
In this section we give concrete examples to illustrate our main results. Consider the following equation
\begin{equation}\label{Ex}
-\operatorname{div}\left( v(|x|) |\nabla u| ^{p-2}\nabla u\right)=\lambda w(|x|)|u|^{p-2}u \quad \text{in } B_1^c
\end{equation}
with $v(|x|)=(|x|-1)^\alpha$ and $w\in L^1_{\loc}(1,\infty)$ such that $w>0$ a.e. in $(1,\infty).$ Note that for such weights $v,w$, the condition ($\mathrm{W}$) is clearly satisfied.
\begin{example}[Degenerate weight]\label{degenexample}\rm
	Let $0\leq \alpha<p-1$.
	\begin{itemize}
		\item If $p \neq N$ and $w \in L^1\big((1,\infty);(r-1)^{p-1-\alpha}\big) \cap L^1\big((1,\infty);(r-1)^{p-1}\big),$ then $v,w$ satisfy $\mathrm{(W_1)}$ of Corollary~\ref{exteriorembedding} and hence,
		$$\mathcal{D}^{1,p}_0(B_1^c; v) \hookrightarrow  \hookrightarrow L^p(B_1^c;w).$$
		In this case, the eigenvalue problem \eqref{Ex} has a principal eigenpair due to Theorem~\ref{Theore.eigenpair}.
		
		\item If $w \in L^1\big((1,\xi);(r-1)^{\delta}\big)$ for some $\xi\in(1,\infty)$ and $\delta<p-1-\alpha,$ then  $\mathrm{(A_{\epsilon,L})}$ holds for $\epsilon\in (\frac{(p-1)\delta}{p-1-\alpha},p-1)\cap (0,\infty)$. By Theorem \ref{radialsoldecayleft}, if $u(x)=u(|x|)\in C^1(B_1^c)$ is a solution to equation \eqref{Ex} with $u(1)=u(\infty)=0$, there exist $a \in (1,\infty)$, $0<C_1<C_2$ and $0<\tilde{C}_1<\tilde{C}_2$, such that 
		\begin{align*}C_1 (r-1)^{\frac{p-1-\alpha}{p-1}} \leq |u(r)| \leq C_2 (r-1)^{\frac{p-1-\alpha}{p-1}}, \ \ \forall r \in (1,a) \text{ and} \\
		\tilde{C}_1(r-1)^{-\frac{\alpha}{p-1}} \leq |u'(r)| \leq \tilde{C}_2 (r-1)^{-\frac{\alpha}{p-1}}, \ \ \forall r \in (1,a).
		\end{align*}
		Since \[ u'_+(1) = \lim _{r \to 1^+}\frac{u(r)-u(1)}{r-1},\] we have $0 <|u'_+(1)| <\infty$ when $\alpha=0$ and $ |u'_+(1)| =\infty$, when $\alpha>0$.
		
		\item If $p<N+\alpha$ and $w \in L^1\big((\bar{\xi},\infty);(r-1)^{\bar{\delta}}\big)$ for some $\bar{\xi}\in(1,\infty)$ and $\bar{\delta}=p-1$ when $\alpha\in (0,p-1)$ and $\bar{\delta}\in(p-1,N-1)$ when $\alpha=0$, then  $\mathrm{(A_{\epsilon,L})}$ holds for some $\epsilon\in (0,p-1).$ By Theorem \ref{radialsoldecayright}, if $u(x)=u(|x|)\in C^1(B_1^c)$ is a solution to equation \eqref{Ex} with $u(1)=u(\infty)=0$,  there exist $b \in (1,\infty)$, $0<C_1<C_2$ and $0<\tilde{C}_1<\tilde{C}_2$, such that 
		\begin{align*}
		C_1r^{-\frac{N-p+\alpha}{p-1}} \leq |u(r)| \leq C_2r^{-\frac{N-p+\alpha}{p-1}}, \ \ \forall r \in (b,\infty), \text{ and }\\
		\tilde{C}_1 r^{-\frac{N-1-\alpha}{p-1}} \leq |u'(r)| \leq \tilde{C}_2 r^{-\frac{N-1-\alpha}{p-1}}, \ \ \forall r \in (b,\infty).
		\end{align*}
	\end{itemize} 
\end{example}
\begin{remark}\rm
For instance, let $v(r)=1$ and $0<w(r)<C r^{-\gamma}$ ($\gamma>p$), we obtain better estimates for $u,u'$ at infinity than that of \cite{Chhetri-Drabek} and \cite{Anoop.CV}, by putting $\alpha=0$ in Example~\ref{degenexample}.
\end{remark}
\begin{example}[Singular weight]\rm
	Consider $1<p<N$ and let $p-N<\alpha<0$. 
\begin{itemize}
	\item If $w \in L^1\big((1,\infty);(r-1)^{p-1}\big) \cap L^1\big((1,\infty);(r-1)^{p-1-\alpha}\big),$ then the weights $v,w$ satisfy $\mathrm{(W_2)}$ of Corollary~\ref{exteriorembedding} and we get
	\[\mathcal{D}^{1,p}_0(B_1^c;v) \hookrightarrow \hookrightarrow L^p(B_1^c;w).\]  
	Hence, the eigenvalue problem \eqref{Ex} has a principal eigenpair due to Theorem~\ref{Theore.eigenpair}.
	
	\item If $w \in L^1\big((1,\xi);(r-1)^{p-1}\big)$ for some $\xi\in(1,\infty)$, then  $\mathrm{(A_{\epsilon,L})}$ holds for $\epsilon\in (\frac{(p-1)^2}{p-1-\alpha},p-1)$. By Theorem \ref{radialsoldecayleft}, if $u(x)=u(|x|)\in C^1(B_1^c)$ is a solution to equation \eqref{Ex} with $u(1)=u(\infty)=0$, there exist $a \in (1,\infty)$, $0<C_1<C_2$ and $0<\tilde{C}_1<\tilde{C}_2$, such that 
	\begin{align*}C_1 (r-1)^{\frac{p-1-\alpha}{p-1}} \leq |u(r)| \leq C_2 (r-1)^{\frac{p-1-\alpha}{p-1}}, \ \ \forall r \in (1,a) \text{ and} \\
	\tilde{C}_1(r-1)^{-\frac{\alpha}{p-1}} \leq |u'(r)| \leq \tilde{C}_2 (r-1)^{-\frac{\alpha}{p-1}}, \ \ \forall r \in (1,a).
	\end{align*}
	In this case, we have $u'_+(1)=0$.
	
	\item If $w \in L^1\big((\bar{\xi},\infty);(r-1)^{\bar{\delta}}\big)$ for some $\bar{\xi}\in(1,\infty)$ and $\bar{\delta}\in(p-1-\alpha,N-1)$, then  $\mathrm{(A_{\epsilon,L})}$ holds for some $\epsilon\in (0,p-1).$ By Theorem \ref{radialsoldecayright}, if $u(x)=u(|x|)\in C^1(B_1^c)$ is a solution to equation \eqref{Ex} with $u(1)=u(\infty)=0$,  there exist $b \in (1,\infty)$, $0<C_1<C_2$ and $0<\tilde{C}_1<\tilde{C}_2$, such that 
	\begin{align*}
	C_1r^{-\frac{N-p+\alpha}{p-1}} \leq |u(r)| \leq C_2r^{-\frac{N-p+\alpha}{p-1}}, \ \ \forall r \in (b,\infty), \text{ and }\\
	\tilde{C}_1 r^{-\frac{N-1-\alpha}{p-1}} \leq |u'(r)| \leq \tilde{C}_2 r^{-\frac{N-1-\alpha}{p-1}}, \ \ \forall r \in (b,\infty).
	\end{align*}
\end{itemize} 
\end{example}

\section*{Acknowledgment} P. Dr\'{a}bek and A. Sarkar were partly supported by the Grant Agency of the Czech Republic, project no. 18-032523S. K. Ho and A. Sarkar were supported by the project LO1506 of the Czech Ministry of Education, Youth and Sports.

\end{document}